\documentclass{amsart}
\usepackage{amscd,amsmath,amssymb}
\usepackage{pstricks}
\usepackage{latexsym,amsbsy,mathrsfs}
\usepackage{xy}
\usepackage{xypic}
\usepackage{pgf,tikz}

\newtheorem{thm}{Theorem}[section]
\newtheorem{lem}[thm]{Lemma}
\newtheorem{cor}[thm]{Corollary}
\newtheorem{prop}[thm]{Proposition}

\theoremstyle{definition}
\newtheorem{example}[thm]{Example}

\newtheorem{defn}[thm]{Definition}

\newtheorem{rem}[thm]{Remark}

\numberwithin{equation}{thm}

\begin{document}
\title[Abelian quotients of the categories of short exact sequences]
{Abelian quotients of the categories of short exact sequences}

\author{Zengqiang Lin}
\address{ School of Mathematical sciences, Huaqiao University,
Quanzhou\quad 362021,  China.} \email{zqlin@hqu.edu.cn}

\thanks{This work was supported  by the Natural Science Foundation of Fujian Province (Grant No. JZ160405)}



\begin{abstract}
We mainly investigate abelian  quotients of the categories of short exact sequences. The natural framework to consider the question is via identifying quotients of morphism categories as modules categories. These ideas not only can be used to recover the abelian quotients produced by cluster-tilting subcategories of both exact categories and triangulated categories, but also can be used to reach our goal. Let $(\mathcal{C},\mathcal{E})$ be an exact category. We denote by $\mathcal{E}(\mathcal{C})$ the category of bounded complexes whose objects are given by short exact sequences in $\mathcal{E}$ and by $S\mathcal{E}(\mathcal{C})$  the full subcategory formed by split short exact sequences. In general, $\mathcal{E}(\mathcal{C})$ is just an exact category, but the quotient $\mathcal{E}(\mathcal{C})/[S\mathcal{E}(\mathcal{C})]$ turns out to be abelian. In particular, if $(\mathcal{C},\mathcal{E})$ is Frobenius, we present three equivalent abelian quotients of $\mathcal{E}(\mathcal{C})$ and point out that the equivalences are actually given by left and right rotations. The abelian quotient $\mathcal{E}(\mathcal{C})/[S\mathcal{E}(\mathcal{C})]$ admits some nice properties. We explicitly describe the abelian structure, projective objects, injective objects and simple objects, which provide a new viewpoint to understanding Hilton-Rees Theorem and Auslander-Reiten theory. Furthermore, we present some analogous results both for $n$-exact versions and for triangulated versions.
\end{abstract}

\maketitle

\section{Introduction}

Cluster-tilting theory provides a way to construct abelian quotient categories. Let $\mathcal{C}$ be a triangulated category and $\mathcal{T}$ be a cluster-tilting subcategory of $\mathcal{C}$, then the quotient $\mathcal{C}/[\mathcal{T}]$ is abelian; related works see \cite{[BMR],[KR],[IY]} and \cite{[KZ]}. The version of exact categories see \cite{[DL]}. Different methods for understanding the abelian quotients have been investigated  further, for example, via localisations \cite{[BM1],[BM2]}, via cotorsion pairs \cite{[Na1],[Na2],[Liu]}, via homotopical algebra \cite{[Pa]} and so on.

 Let $\mathcal{C}$ be an abelian category. Denote by $\mathcal{E}(\mathcal{C})$ the category of all short exact sequences in $\mathcal{C}$. It is well known that $\mathcal{E}(\mathcal{C})$ is an exact category but it is not abelian in general. Denote by Mor$(\mathcal{C})$ the morphism category of $\mathcal{C}$, by Mono$(\mathcal{C})$ the monomorphism category of $\mathcal{C}$, and by Epi$(\mathcal{C})$ the epimorphism category of $\mathcal{C}$. Then the three categories $\mathcal{E}(\mathcal{C})$, Mono$(\mathcal{C})$ and Epi$(\mathcal{C})$ are equivalent. Note that in the case when $\mathcal{C}$ is the module category over a ring, then the monomorphism category Mono$(\mathcal{C})$ is known as the submodule category. The structure of submodule categories has been studied intensively by Ringel and Schmidmeier \cite{[RS1],[RS2]}. Let $\mathcal{S}(A)$ be the submodule category of an artin $k$-algebra $A$. If $A=k[t]/\langle t^n\rangle$, Ringel and Zhang established two abelian quotients of $\mathcal{S}(A)$ \cite[Theorem 1]{[RZ]}.  Denote by $\mathcal{U}_1$ (resp. $\mathcal{U}_2$) the full subcategory of $\mathcal{S}(A)$ formed by objects of the form $(X\xrightarrow{1} X)\oplus (0\rightarrow Y)$ (resp. $(X\xrightarrow{1} X)\oplus (Y\rightarrow P)$ with $P$ projective-injective). They showed that the quotient categories $\mathcal{S}(A)/[\mathcal{U}_1]$ and $\mathcal{S}(A)/[\mathcal{U}_2]$ are equivalent to mod-$\Pi_{n-1}$ where $\Pi_{n-1}$ is the preprojective algebra of type $A_{n-1}$. Recently due to Eir\'{\i}ksson  \cite[Theorem1]{[Eir]}, the above result was generalized for any self-injective algebra of finite representation type by replacing $\Pi_{n-1}$ with $\underline{B}$,  the stable Auslander algebra of $A$.

%Note that in the case when $\mathcal{C}$ is the module category over a ring, then the monomorphism category Mono$(\mathcal{C})$ is known as the submodule category. %which has been studied intensively by Ringel and Schmidmeier.
% The equivalences established in \cite{[RZ],[Eir]} propose the following question.
The present paper mainly studies the abelian quotients of the categories of short exact sequences. Our approach to understanding abelian quotients is via morphism categories. The following is a basic proposition.

\begin{prop}
Let $\mathcal{C}$ be an additive category, then we have the following equivalences.

\textup{(a)} $\textup{Mor}(\mathcal{C})/[\mathcal{U}]\cong \textup{mod-}\mathcal{C}$, where $\mathcal{U}$ is the full subcategory of $\textup{Mor}(\mathcal{C})$ consisting of $(X\xrightarrow{1} X)\oplus (Y\rightarrow 0)$.

\textup{(b)} $\textup{Mor}(\mathcal{C})/[\mathcal{U'}]\cong (\textup{mod-}\mathcal{C}^{\textup{op}})^{\textup{op}}$, where $\mathcal{U'}$ is the full subcategory of $\textup{Mor}(\mathcal{C})$ consisting of $(X\xrightarrow{1} X)\oplus (0\rightarrow Y)$.
\end{prop}

Using Proposition 1.1, we realize some abelian quotient categories constructed by cluster-tilting subcategories. For example, we can reprove \cite[Theorem 3.2]{[DL]}, \cite[Theorem 3.5]{[DL]} and \cite[Corollary 4.4]{[KZ]}.

Let $(\mathcal{C},\mathcal{E})$ be an exact category. We denote by $C^b(\mathcal{C})$ the category of bounded complexes over $\mathcal{C}$, by $\mathcal{E}(\mathcal{C})$ the full subcategory of $C^b(\mathcal{C})$ consisting of short exact sequences in $\mathcal{E}$, by $S\mathcal{E}(\mathcal{C})$ the full subcategory of $\mathcal{E}(\mathcal{C})$ formed by split short exact sequences over $\mathcal{C}$. A short exact sequence $0\rightarrow X_1\xrightarrow{f_1} X_2\xrightarrow{f_2} X_3\rightarrow 0$ is denoted by $(X_1\rightarrow X_2\rightarrow X_3)$ for short. The following is our main theorem.

\begin{thm}\label{thm1.1}
Let $(\mathcal{C},\mathcal{E})$ be an exact category and $X_\bullet:0\rightarrow X_1\xrightarrow{f_1} X_2\xrightarrow{f_2} X_3\rightarrow 0$ be a short exact sequence in $\mathcal{E}$.

\textup{(a)} If $(\mathcal{C},\mathcal{E})$ has enough projectives, denote by $\mathcal{P}$ the full subcategory of $\mathcal{C}$ formed by all projectives,
 then we have the following equivalences:
$$\alpha_1:\mathcal{E}(\mathcal{C})/[S\mathcal{E(}\mathcal{C})]\cong \textup{mod-}\mathcal{C}/[\mathcal{P}],\ \ X_\bullet \mapsto \textup{Coker}(\mathcal{C}/[\mathcal{P}](-,f_2))$$
$$\alpha_2:\mathcal{E}(\mathcal{C})/[P\mathcal{E}(\mathcal{C})]\cong (\textup{mod-}(\mathcal{C}/[\mathcal{P}])^{\textup{op}})^{\textup{op}}, \ \ X_\bullet \mapsto \textup{Coker}(\mathcal{C}/[\mathcal{P}](f_2,-))$$ where $P\mathcal{E}(\mathcal{C})$ is the full subcategory of $\mathcal{E}(\mathcal{C})$ formed by $(0\rightarrow X\rightarrow X)\oplus(\Omega Y\rightarrow P\rightarrow Y)$.

\textup{(b)} If $(\mathcal{C},\mathcal{E})$ has enough injectives, denote by $\mathcal{I}$ the full subcategory of $\mathcal{C}$ formed by all injectives,
then we have the following equivalences:
$$\beta_1:\mathcal{E}(\mathcal{C})/[S\mathcal{E}(\mathcal{C})]\cong (\textup{mod-}(\mathcal{C}/[\mathcal{I}])^{\textup{op}})^{\textup{op}}, \ \ X_\bullet\mapsto \textup{Coker}(\mathcal{C}/[\mathcal{I}](f_1,-))$$
$$\beta_2:\mathcal{E}(\mathcal{C})/[I\mathcal{E}(\mathcal{C})]\cong \textup{mod-}\mathcal{C}/[\mathcal{I}], \ \ X_\bullet\mapsto\textup{Coker}(\mathcal{C}/[\mathcal{I}](-,f_1))$$ where $I\mathcal{E}(\mathcal{C})$ is the full subcategory of $\mathcal{E}(\mathcal{C})$ formed by $(X\rightarrow X\rightarrow 0)\oplus( Y\rightarrow I\rightarrow \Omega^{-1}Y)$.
\end{thm}

In particular, if $(\mathcal{C},\mathcal{E})$ is a Frobenius category, then  the quotient categories $\mathcal{E}(\mathcal{C})/[S\mathcal{E}(\mathcal{C})]$, $\mathcal{E}(\mathcal{C})/[P\mathcal{E}(\mathcal{C})]$  and $\mathcal{E}(\mathcal{C})/[I\mathcal{E}(\mathcal{C})]$ are equivalent to abelian category $\textup{mod-}\mathcal{C}/[\mathcal{P}]$. As we can see in Remark \ref{rem4.1}, the equivalences between the three quotient categories are given by left rotations and right rotations of short exact sequences.

If $(\mathcal{C},\mathcal{E})$ is Frobenius, we can show that the three categories $\mathcal{E}(\mathcal{C})$, Mono$(\mathcal{C})$ and Epi$(\mathcal{C})$ are equivalent. Therefore, we have the following result, which generalizes \cite[Theorem 1]{[Eir]}.

\begin{cor}
Let $(\mathcal{C},\mathcal{E})$ be a Frobenius  category. Denote by $\mathcal{P}$ the full subcategory of projective-injective objects in $\mathcal{C}$, by $\mathcal{U}_1$ the full subcategory of $\textup{Mono}(\mathcal{C})$ consisting of $(X\xrightarrow{1} X)\oplus (0\rightarrow Y)$, by $\mathcal{U}_2$ the full subcategory of $\textup{Mono}(\mathcal{C})$ consisting of $(X\xrightarrow{1} X)\oplus (Y\rightarrow P)$ with $P\in\mathcal{P}$ and by $\mathcal{U}_3$ the full subcategory of $\textup{Mono}(\mathcal{C})$ consisting of $(0\rightarrow X)\oplus (Y\rightarrow P)$ with $P\in\mathcal{P}$. Then each of the quotient categories $\textup{Mono}(\mathcal{C})/[\mathcal{U}_1]$, $\textup{Mono}(\mathcal{C})/[\mathcal{U}_2]$  and $\textup{Mono}(\mathcal{C})/[\mathcal{U}_3]$ is equivalent to $\textup{mod-}\mathcal{C}/[\mathcal{P}]$.
\end{cor}

Our second part of the paper is to studying the properties of $\mathcal{E}(\mathcal{C})/[S\mathcal{E(}\mathcal{C})]$. We show that the abelian structure is given by pullback and pushout diagrams; see Theorem \ref{thm4.3.1}. We characterize the simple objects in $\mathcal{E}(\mathcal{C})/[S\mathcal{E(}\mathcal{C})]$ as the Auslander-Reiten sequences in $\mathcal{C}$; see Theorem \ref{thm4.4.1}. We describe the projective objects and injective objects in $\mathcal{E}(\mathcal{C})/[S\mathcal{E}(\mathcal{C})]$; see Proposition \ref{prop1}. In particular, if $(\mathcal{C},\mathcal{E})$ has enough projectives, then each projective object in $\mathcal{E}(\mathcal{C})/[S\mathcal{E}(\mathcal{C})]$ is of the form
$P_X: 0\rightarrow \Omega X\rightarrow P\rightarrow X\rightarrow 0$ for some object $X$ in $\mathcal{C}$.

As applications, our results provide a new viewpoint to understanding Hilton-Reees Theorem and Auslander-Reiten theory.  Now we assume that $(\mathcal{C},\mathcal{E})$ is an exact category with enough projectives and injectives. By Theorem \ref{thm1.1}, we have a duality $$\Phi:\mbox{mod-}\mathcal{C}/[\mathcal{P}]\rightarrow\mbox{mod-}(\mathcal{C}/[\mathcal{I}])^{\textup{op}},\ \ \ \delta^\ast\mapsto \delta_\ast$$
where $\delta$ is a short exact sequence in $\mathcal{E}$, $\delta^\ast$ is the contravariant defect and $\delta_\ast$ is the covariant defect.
 Moreover, by restrictions and Proposition \ref{prop1}, we obtain the following two dualities
$$\Phi:\mbox{proj-}\mathcal{C}/[\mathcal{P}]\rightarrow\mbox{inj-}(\mathcal{C}/[\mathcal{I}])^{\textup{op}}, \ \ \ \mathcal{C}/[\mathcal{P}](-,X)\mapsto \mbox{Ext}^1_\mathcal{C}(X,-).$$
$$\Phi:\mbox{inj-}\mathcal{C}/[\mathcal{P}]\rightarrow\mbox{proj-}(\mathcal{C}/[\mathcal{I}])^{\textup{op}}, \ \ \ \mbox{Ext}^1_\mathcal{C}(-,X)\mapsto\mathcal{C}/[\mathcal{I}](X,-).$$
Hence, the following result seems natural.

\begin{thm} (Hilton-Rees Theorem, see \cite{[HR],[Ma]}) \label{thm1.4}
Let $(\mathcal{C},\mathcal{E})$ be an exact category with enough projectives and injectives.

\textup{(a)} There is an isomorphism between $\mathcal{C}/[\mathcal{P}](Y,X)$ and the group of natural transformations from $\textup{Ext}^1_\mathcal{C}(X,-)$ to $\textup{Ext}^1_\mathcal{C}(Y,-)$.

\textup{(b)} There is an isomorphism between $\mathcal{C}/[\mathcal{I}](X,Y)$ and the group of natural transformations from $\textup{Ext}^1_\mathcal{C}(-,X)$ to $\textup{Ext}^1_\mathcal{C}(-,Y)$.
\end{thm}

If furthermore, $\mathcal{C}$ is a dualizing $k$-variety, then $\mathcal{C}/[\mathcal{P}]$ and $\mathcal{C}/[\mathcal{I}]$ are also dualizing $k$-varieties. Thus we have two dualities $\Phi:\mbox{mod-}\mathcal{C}/[\mathcal{P}]\rightarrow \mbox{mod-}(\mathcal{C}/[\mathcal{I}])^{\textup{op}}$ and $D:\mbox{mod-}(\mathcal{C}/[\mathcal{I}])^{\textup{op}}\rightarrow\mbox{mod-}\mathcal{C}/[\mathcal{I}]$. The composition of $\Phi$ and $D$ defines an equivalence
$$\Theta: \ \mbox{proj-}\mathcal{C}/[\mathcal{P}] \xrightarrow{\Phi}  \mbox{inj-}(\mathcal{C}/[\mathcal{I}])^{\textup{op}} \xrightarrow{D}\mbox{proj-}\mathcal{C}/[\mathcal{I}].$$
Now we have the following generalized Auslander-Reiten duality and  defect formula.

\begin{thm}\label{thm1.5}
Let $(\mathcal{C},\mathcal{E})$ be an Ext-finite exact category with enough projectives and injectives. Assume that $\mathcal{C}$ is a dualizing $k$-variety. Then there is an equivalence $\tau:\mathcal{C}/[\mathcal{P}]\cong\mathcal{C}/[\mathcal{I}]$ satisfying the following properties:

\textup{(a)} $D\textup{Ext}^1_\mathcal{C}(-,X)\cong\mathcal{C}/[\mathcal{P}](\tau^{-1} X,-)$, $D\textup{Ext}^1_\mathcal{C}(X,-)\cong\mathcal{C}/[\mathcal{I}](-,\tau X)$.

\textup{(b)} $D\delta_\ast=\delta^\ast\tau^{-1}$, $D\delta^\ast=\delta_\ast\tau$ for each short exact sequence $\delta$ in $\mathcal{E}$.

\noindent Therefore, $\mathcal{C}$ has Auslander-Reiten sequences.
\end{thm}

We point out that for $n$-exact categories and triangulated categories, by considersing quotients of the categories of $n$-exact sequences and quotients of the categories of triangles, we obtain some analogous results.

This paper is organized as follows.

In Section 2, we make some preliminaries. We collect some definitions and facts on morphism categories, exact categories, quotient categories and functor categories.

In Section 3, we provide techniques to identify quotients of morphism categories as module categories. Subsection 3.1 is devoted to proving Proposition 1.1. In subsection 3.2, we apply  Proposition 1.1 to exact categories. We show that some quotient categories of epimorphism categories are equivalent to module categories, see Theorem \ref{thm3.2.1}, which can be used to prove \cite[Theorem 3.2]{[DL]} and \cite[Theorem 3.2]{[DL]}; see Corollary \ref{cor3.1}. We obtain certain recollements of abelian categories from the viewpoint of morphism categories, which implies Auslander's formula; see Corollary \ref{cor3.2}. In subsection 3.3, we apply Proposition 1.1 to triangulated categories; see Proposition \ref{prop3.0} and Corollary \ref{cor3.3}. In subsection 3.4, we give some examples.

In Section 4, we study the abelian quotients of the categories of short exact sequences. In subsection 4.1 we realize some quotients of these categories as module categories; see Theorem \ref{thm1.1}. In subsection 4.2 we describe the abelian structure of the quotients; see Theorem \ref{thm4.3.1}. In subsection 4.3 we study the projective objects and injective objects, which are applied to prove Hilton-Rees Theorem; see Proposition \ref{prop1} and Theorem \ref{thm1.4}. In subsection 4.4, we will restrict our attention to the connection to Auslander-Reiten theory. We will prove Theorem \ref{thm4.4.1} and Theorem \ref{thm1.5}. Subsection 4.5 is devoted to listing the higher versions on the abelian quotients of the categories of $n$-exact sequences.

In Section 5, we consider the abelian quotients of the categories of triangles. % in a triangulated category $\mathcal{C}$.
There are some parallel results. %see Theorem \ref{thm5.1}, Theorem \ref{thm5.2} and Theorem \ref{thm5.3}.

\section{Preliminaries}

In this section, we make some preliminaries.
Let $\mathcal{C}$ be an additive category. We denote by $\mathcal{C}(X,Y)$ the set of morphisms from $X$ to $Y$ in $\mathcal{C}$. The composition of $f\in\mathcal{C}(X,Y)$ and $g\in\mathcal{C}(Y,Z)$ is denoted by $gf$.

\subsection{Morphism categories}
Assume that $\mathcal{C}$ is an additive category. The {\em morphism category} of $\mathcal{C}$ is the category $\textup{Mor}(\mathcal{C})$ defined by the following data.
The objects of $\textup{Mor}(\mathcal{C})$ are all  the morphisms $f:X\rightarrow Y$ in $\mathcal{C}$. The morphisms from $f:X\rightarrow Y$ to $f':X'\rightarrow Y'$ are pairs $(a,b)$ where $a:X\rightarrow X'$ and $b:Y\rightarrow Y'$ such that $bf=f'a$. The composition of morphisms is componentwise. We denote by $\textup{Mono}(\mathcal{C})$ the full subcategory of $\textup{Mor}(\mathcal{C})$ consisting of monomorphisms in $\mathcal{C}$, which is called the {\em monomorphism category of $\mathcal{C}$}. Dually, we define {\em epimorphism category} $\textup{Epi}(\mathcal{C})$ of $\mathcal{C}$. In particular, if $\mathcal{C}$ is abelian, then $\textup{Mor}(\mathcal{C})$ is an abelian category. In this case, $\textup{Mono}(\mathcal{C})$ is an additive category of $\textup{Mor}(\mathcal{C})$ which is closed under extensions, thus it becomes an exact category. Moreover, $\textup{Mono}(\mathcal{C})$ is isomorphic to $\textup{Epi}(\mathcal{C})$, where the isomorphism is given by cokernel functor.

\subsection{Exact categories}

We recall the notion of exact categories from \cite{[B]}.
Let $\mathcal{C}$ be an additive category. A {\em kernel-cokernel pair} $(i,p)$ in $\mathcal{C}$ is a pair of composable morphisms $X\xrightarrow{i} Y\xrightarrow{p} Z$ such that $i$ is a kernel of $p$ and $p$ is a cokernel of $i$. Assume that $\mathcal{E}$ is a class of kernel-cokernel pairs. A kernel-cokernel pair $(i,p)$ in $\mathcal{E}$ is called a {\em short exact sequence} in $\mathcal{E}$, which is denoted by $0\rightarrow X\xrightarrow{i} Y\xrightarrow{p} Z\rightarrow 0$.   A morphism $p:Y\rightarrow Z$ is called {\em admissible epimorphism} if there exists a morphism $i:X\rightarrow Y$ such that $(i,p)\in\mathcal{E}$. {\em Admissible monomorphisms} are defined dually.

A class of  kernel-cokernel pairs $\mathcal{E}$ is called an {\em exact structure} of $\mathcal{C}$ if $\mathcal{E}$ is closed under isomorphisms and satisfies the following axioms:

(E0) Identity morphisms are admissible epimorphisms.

(E0)$^{\textup{\tiny {op}}}$ Identity morphisms are  admissible monomorphisms.

(E1) The composition of two admissible epimorphisms  is an admissible epimorphism.

(E1)$^{\textup{\tiny {op}}}$ The composition of two admissible monomorphisms is an admissible monomorphism.

(E2) Given a short exact sequence  $0\rightarrow X\xrightarrow{i} Y\xrightarrow{p} Z\rightarrow 0$ in $\mathcal{E}$ and a morphisms $\varphi:X\rightarrow X'$ in $\mathcal{C}$, there exists a commutative diagram
$$\xymatrix
{0\ar[r] & X\ar[r]^{i}\ar[d]^{\varphi} & Y\ar[r]^{p}\ar[d]^{\varphi'} & Z\ar[r] \ar@{=}[d] & 0\\
0\ar[r] & X'\ar[r]^{i'} & Y' \ar[r]^{p'} & Z\ar[r] & 0
}$$
such that the second row belongs to $\mathcal{E}$. In this case, $(0\rightarrow X\xrightarrow{\left(
                                                                                  \begin{smallmatrix}
                                                                                    i \\
                                                                                    \varphi \\
                                                                                  \end{smallmatrix}
                                                                                \right)
} Y\oplus X'\xrightarrow{(\varphi',-i')} Y'\rightarrow 0)\in\mathcal{E}$.

(E2)$^{\textup{\tiny {op}}}$  Given a short exact sequence  $0\rightarrow X\xrightarrow{i} Y\xrightarrow{p} Z\rightarrow 0$ in $\mathcal{E}$ and a morphisms $\phi: Z'\rightarrow Z$ in $\mathcal{C}$, there exists a commutative diagram
$$ \xymatrix
{0\ar[r] & X\ar[r]^{i'}\ar@{=}[d] & Y'\ar[r]^{p'} \ar[d]^{\phi'} & Z'\ar[r] \ar[d]^{\phi} & 0\\
0\ar[r] & X\ar[r]^{i} & Y \ar[r]^{p} & Z\ar[r] & 0
}$$
such that the first row belongs to $\mathcal{E}$. In this case, $(0\rightarrow Y'\xrightarrow{\left(
                                                                                  \begin{smallmatrix}
                                                                                    p' \\
                                                                                    \phi' \\
                                                                                  \end{smallmatrix}
                                                                                \right)
} Z'\oplus Y\xrightarrow{(\phi,-p)} Z\rightarrow 0)\in\mathcal{E}$.

An {\em exact category} is an additive category $\mathcal{C}$ admits an exact structure $\mathcal{E}$, which is denoted by $(\mathcal{C},\mathcal{E})$.

For example, an additive category is an exact category with respect to the class of split short exact sequences, which are isomorphic to $0\rightarrow X\xrightarrow{\left(
                \begin{smallmatrix}
                  1 \\
                  0 \\
                \end{smallmatrix}
              \right)
} X\oplus Y\xrightarrow{(0,1)} Y\rightarrow 0$ for some $X,Y\in\mathcal{C}$. An abelian category $\mathcal{C}$ is an exact category where the exact structure is given by all the kernel-cokernel pairs in $\mathcal{C}$.

An object $P$ of an exact category $(\mathcal{C},\mathcal{E})$ is called {\em projective} if for each admissible epimorphism $p:Y\rightarrow Z$ and each morphism $f:P\rightarrow Z$, there exists a morphism $g:P\rightarrow Y$ such that $f=pg$. The full subcategory of projectives is denoted by $\mathcal{P}$. We say an exact category $(\mathcal{C},\mathcal{E})$ {\em has enough projective objects} if for each object $X\in\mathcal{C}$ there is an admissible epimorphism $p: P\rightarrow X$ with $P\in\mathcal{P}$. Dually, we can define injective objects. The full subcategory of injectives is denoted by $\mathcal{I}$. An exact category is {\em Frobenius} provided that it has enough projectives and injectives and, moreover, the classes of projectives and injectives coincide. If an exact category $(\mathcal{C},\mathcal{E})$ has enough projectives, then we can consider the projective resolutions and Ext functors as right derived functors of Hom as in abelian categories. Hence, Ext$^1_\mathcal{C}(Z,X)$ parameterizes the short exact sequences $0\rightarrow X\rightarrow Y\rightarrow Z\rightarrow 0$ in $\mathcal{E}$ up to equivalence.

\subsection{Quotient categories}

Let $\mathcal{C}$ be an additive category. An {\em {ideal}} $\mathcal{I}$ of $\mathcal{C}$ is a class of additive subgroups $\mathcal{I(}X,Y)$ of $\mathcal{C(}X,Y)$ such that $hgf\in \mathcal{I}(X,W)$ for each $f\in\mathcal{C}(X,Y), g\in \mathcal{I}(Y,Z)$ and $h\in\mathcal{C}(Z,W)$.  Assume that $\mathcal{I}$ is an ideal of $\mathcal{C}$, then by definition, the {\em quotient category} $\mathcal{C}/\mathcal{I}$ has the same objects as $\mathcal{C}$ and has morphisms $\mathcal{C}/\mathcal{I}(X,Y)=\mathcal{C}(X,Y)/\mathcal{I}(X,Y)$. For example, the Jacobson radical $J_\mathcal{C}$ of $\mathcal{C}$ is an ideal of $\mathcal{C}$. Suppose that $\mathcal{D}$ is a full subcategory of $\mathcal{C}$. We denote by $[\mathcal{D}](X,Y)$ the subset of morphisms of $\mathcal{C}(X,Y)$ which factor through an object in $\mathcal{D}$. It is easy to see that $[\mathcal{D}]$ is an ideal of $\mathcal{C}$, thus we have a quotient category $\mathcal{C}/[\mathcal{D}]$ and a quotient functor $Q:\mathcal{C}\rightarrow\mathcal{C}/[\mathcal{D}]$. Let $f:X\rightarrow Y$ be a morphism in $\mathcal{C}$. The image of $f$ under $Q$ is denoted by $\underline{f}$. It is well known that for each additive functor $F:\mathcal{C}\rightarrow\mathcal{E}$, if $F(\mathcal{D})=0$, then there is a unique functor $F':\mathcal{C}/[\mathcal{D}]\rightarrow\mathcal{E}$ such that $F'Q=F$.

Let $F:\mathcal{C}\rightarrow \mathcal{D}$ be a full and dense functor. If each morphism $f\in\mathcal{C}(X,Y)$ with $F(f)=0$ factors through an object $Z$ with $F(Z)=0$, then the functor $F$ is called {\em objective} (see \cite{[RZ]}). In this case, there is an equivalence $\mathcal{C}/[\textup{Ker}F]\cong \mathcal{D}$, where $\textup{Ker}F$ is the full subcategory of $\mathcal{C}$ formed by $X$ with $F(X)=0$.

\subsection{Functor categories}
Let $\mathcal{C}$ be an additive category. A right {\em $\mathcal{C}$-module} is a contravariantly additive functor $F:\mathcal{C}\rightarrow Ab$ where $Ab$ is the category of abelian groups.  Denote by Mod-$\mathcal{C}$  the category of right $\mathcal{C}$-modules.  It is well known that Mod-$\mathcal{C}$ is an abelian category. The $\mathcal{C}$-module $\mathcal{C}(-,X)$ is a projective object of Mod-$\mathcal{C}$ for each object $X\in\mathcal{C}$. Moreover, each projective object is a direct summand of $\mathcal{C}(-,X)$ for some $X\in\mathcal{C}$. By definition, a $\mathcal{C}$-module $F$ is called {\em{finitely presented}} (or {\em coherent}) if there exists an exact sequence $\mathcal{C}(-,X)\rightarrow\mathcal{C}(-,Y)\rightarrow F\rightarrow 0$. We denote by mod-$\mathcal{C}$ the full subcategory of Mod-$\mathcal{C}$ formed by finitely presented $\mathcal{C}$-modules,  by proj-$\mathcal{C}$ (resp. inj-$\mathcal{C}$) the full subcategory of mod-$\mathcal{C}$ consisting of projective (resp. injective) objects. It is known that mod-$\mathcal{C}$ is closed under cokernels and extensions. Moreover, we have the following result.

\begin{prop} (\cite{[Aus]}, \cite[Lemma 4.1]{[Kr]})
Let $\mathcal{C}$ be an additive category. Then $\textup{mod-}\mathcal{C}$ is abelian if and only if $\mathcal{C}$ admits weak kernels.
\end{prop}

Recall that a morphism $f:X\rightarrow Y$ in $\mathcal{C}$ is a {\em weak kernel} of $g: Y\rightarrow Z$ if $gf=0$ and for each morphism $h:W\rightarrow Y$ such that $gh=0$, there exists a morphism $p:W\rightarrow X$ such that $fp=h$.

\begin{rem}\label{rem2.2}
Assume that  $\mathcal{C}$ admits weak kernels. For later use, we recall the abelian structure of mod-$\mathcal{C}$. Let $\alpha: F_1\rightarrow F_2$ be a morphism in mod-$\mathcal{C}$ with the following presentation:
$$\xymatrix{\mathcal{C}(-,X_1)\ar[r] \ar[d] & \mathcal{C}(-,Y_1)\ar[r]\ar[d] & F_1 \ar[r]\ar[d]^{\alpha} & 0\\
\mathcal{C}(-,X_2)\ar[r]  & \mathcal{C}(-,Y_2)\ar[r] & F_2 \ar[r] & 0
}$$ Then Coker$(\mathcal{C}(-,Y_1\oplus X_2)\rightarrow \mathcal{C}(-,Y_2))$ is a cokernel of $\alpha$. Suppose that $Z_1\rightarrow Y_1\oplus X_2$ is a weak kernel of $Y_1\oplus X_2\rightarrow Y_2$ and  $Z_2\rightarrow Z_1\oplus X_1$ is a weak kernel of $Z_1\oplus X_1\rightarrow Y_1$, then  Coker$(\mathcal{C}(-,Z_2)\rightarrow \mathcal{C}(-,Z_1))$ is a kernel of $\alpha$.
\end{rem}

Let $\mathcal{D}$ be a full subcategory of $\mathcal{C}$. A morphism $f:D\rightarrow X$ is called a {\em right $\mathcal{D}$-approximation} of $X$ if $D\in \mathcal{D}$ and each morphism $g:D'\rightarrow X$ with $D'\in\mathcal{D}$ factors through $f$. The category $\mathcal{D}$ is called {\em contravariantly finite} if each object in $\mathcal{C}$ admits a right $\mathcal{D}$-approximation. A contravariantly finite and covariantly finite subcategory is called {\em functorially finite}.

\begin{example}
(a) Let $\mathcal{C}$ be an abelian category, then mod-$\mathcal{C}$ is abelian.

(b) Let $\mathcal{C}$ be an exact category with enough projectives. Denote by $\mathcal{P}$ the subcategory of projectives. If $\mathcal{M}$ is a contravariantly finite subcategory of $\mathcal{C}$, then mod-$\mathcal{M}$ is abelian. Moreover, if $\mathcal{M}$ contains $\mathcal{P}$, then mod-$\mathcal{M}/[\mathcal{P}]$ is still abelian (see \cite[Lemma 2.3]{[DL]}). In particular, mod-$\mathcal{C}/[\mathcal{P}]$ is abelian.

(c) Let $\mathcal{C}$ be a triangulated category, then mod-$\mathcal{C}$ is abelian.
\end{example}

The following result generalizes \cite[Proposition 4.1]{[AHK]} slightly.

\begin{prop}\label{prop2.1}
Let $\mathcal{C}$ be an additive category and $\mathcal{D}$ be a contravariently finite subcategory. Then

\textup{(a)} $\textup{Mod-}\mathcal{C}/[\mathcal{D}]\cong \{F\in\textup{Mod-}\mathcal{C}|F(\mathcal{D})=0\}.$

\textup{(b)} $ \textup{mod-}\mathcal{C}/[\mathcal{D}]\cong \textup{mod-}_0\mathcal{C}=\{F\in\textup{mod-}\mathcal{C}|F(\mathcal{D})=0\}.$
\end{prop}

\begin{proof} (a) It follows from the universal property of quotient functors. For convenience, we identify Mod-$\mathcal{C}/[\mathcal{D}]$ and $\{F\in\textup{Mod-}\mathcal{C}|F(\mathcal{D})=0\}$.

(b) For each object $X\in\mathcal{C}$, we assume that $f:D\rightarrow X$ is a right $\mathcal{D}$-approximation of $X$. Since Im$\mathcal{C}(-,f)=[\mathcal{D}](-,X)$, we have the following exact sequence
$$\mathcal{C}(-,D)\xrightarrow{\mathcal{C}(-,f)}\mathcal{C}(-,X)\rightarrow \mathcal{C}(-,X)/[\mathcal{D}](-,X)\rightarrow 0.$$
It follows that $\mathcal{C}/[\mathcal{D}](-,X)\in\textup{mod-}\mathcal{C}$, since $\mathcal{C}/[\mathcal{D}](Y,X)=\mathcal{C}(Y,X)/[\mathcal{D}](Y,X)$ for each $Y\in\mathcal{C}$. Thus $\mathcal{C}/[\mathcal{D}](-,X)\in\textup{mod-}_0\mathcal{C}$ since $f$ is a right $\mathcal{D}$-approximation. Consequently, mod-$\mathcal{C}/[\mathcal{D}]\subseteq \textup{mod-}_0\mathcal{C}$. On the other hand, for each $F\in\textup{mod-}_0\mathcal{C}$, there is an exact sequence $\mathcal{C}(-,X_1)\rightarrow\mathcal{C}(-,X_2)\rightarrow F\rightarrow 0$ with $F(\mathcal{D})=0$.
The following exact sequence
$$\mathcal{C}(-,X_1)/[\mathcal{D}](-,X_1)\rightarrow
\mathcal{C}(-,X_2)/[\mathcal{D}](-,X_2)\rightarrow F\rightarrow 0$$
 shows that $F\in \textup{mod-}\mathcal{C}/[\mathcal{D}]$.
\end{proof}

Let $k$ be a commutative artinian ring and $E$ be the injective envelope of $k$. Set $D=\textup{Hom}_k(-,E)$. A $k$-linear additive category $\mathcal{C}$ is called {\em dualizing $k$-variety} if the functor $D:\textup{Mod-}\mathcal{C}\rightarrow \textup{Mod-}\mathcal{C}^{\textup{op}}$ given by $D(F)(X):=D(F(X))$, induces a duality $D:\textup{mod-}\mathcal{C}\rightarrow \textup{mod-}\mathcal{C}^{\textup{op}}$.

\begin{example}\label{ex2.1}
(a) Let $A$ be an artin $k$-algebra. Denote by mod-$A$ the category of finitely presented right $A$-modules, and by proj-$A$ the full subcategory of mod-$A$ formed by projective $A$-modules. Then both mod-$A$ and proj-$A$ are dualizing $k$-varieties.

(b)  Let $\mathcal{C}$ be a dualizing $k$-variety, then mod-$\mathcal{C}$ is a dualizing $k$-variety. Moreover, mod-$\mathcal{C}$ is an abelian category with enough projectives and enough injectives.

(c) Any functorially finite subcategory of a dualizing $k$-variety is also a dualizing $k$-variety.

(d) Let $\mathcal{C}$ be a dualizing $k$-variety and $\mathcal{D}$ be a contravariantly finite subcategory, then $\mathcal{C}/[\mathcal{D}]$ is a dualizing $k$-variety.
\end{example}

\begin{proof}
Since one can find (a) and (b) in \cite{[AR]} and find (c) in \cite{[AS]}, we only prove (d).
Let $F\in\textup{mod-}\mathcal{C}/[\mathcal{D}]$, then by Proposition \ref{prop2.1}, $DF$ can be viewed as a finitely presented $\mathcal{C}^{\textup{op}}$-module which vanishes on $\mathcal{D}$. Thus $DF\in\textup{mod-}(\mathcal{C}/[\mathcal{D}])^{\textup{op}}$. Conversely, we can show that if $F\in\textup{mod-}(\mathcal{C}/[\mathcal{D}])^{\textup{op}}$, then $DF\in\textup{mod-}\mathcal{C}/[\mathcal{D}]$.
\end{proof}

\section{Identifying quotients of morphism categories as module categories}

Our approach to understanding the categories of short exact sequences will be based on viewing them as morphism categories, which we are able to identify with certain module categories. In this section we provide techniques needed for such identifications.

\subsection{Basic case: additive categories}
Let $\mathcal{C}$ be an additive category. For two objects $f:X\rightarrow Y$ and $f':X'\rightarrow Y'$ in $\textup{Mor}(\mathcal{C})$, we define $\mathcal{R}(f,f')$ (resp. $\mathcal{R'}(f,f'))$ to be the set of morphisms $(a,b)$ such that there is some morphism $p:Y\rightarrow X'$ such that $f'p=b$ (resp. $pf=a$). Then $\mathcal{R}$ and $\mathcal{R}'$ are ideals of $\textup{Mor}(\mathcal{C})$.

\begin{lem}\label{prop3.1}
Let $\mathcal{C}$ be an additive category, then we have the following equivalences.

\textup{(a)} $\textup{Mor}(\mathcal{C})/\mathcal{R}\cong \textup{mod-}\mathcal{C}$.

\textup{(b)} $\textup{Mor}(\mathcal{C})/\mathcal{R'}\cong (\textup{mod-}\mathcal{C}^{\textup{op}})^{\textup{op}}$.
\end{lem}

\begin{proof}  (a) We define a functor $\alpha:\textup{Mor}(\mathcal{C})\rightarrow \textup{mod-}\mathcal{C}$ by mapping $f:X\rightarrow Y$ to $F=\textup{Coker}(\mathcal{C}(-,f):\mathcal{C}(-,X)\rightarrow \mathcal{C}(-,Y))$. The functor $\alpha$ is dense and full by Yoneda's lemma. Suppose that $(a,b)$ is a morphism from $f:X\rightarrow Y$ to $f':X'\rightarrow Y'$. If $\alpha(a,b)=0$, then the following diagram
$$\xymatrix{
\mathcal{C}(-,X)\ar[r]^{\mathcal{C}(-,f)}\ar[d]^{\mathcal{C}(-,a)} & \mathcal{C}(-,Y)\ar[r]\ar[d]^{\mathcal{C}(-,b)} & F \ar[r]\ar[d]^{0} & 0 \\
\mathcal{C}(-,X')\ar[r]^{\mathcal{C}(-,f')} & \mathcal{C}(-,Y')\ar[r] & F'\ar[r] & 0 \\
}$$ is commutative and each row is exact. There exists a morphism $\mathcal{C}(-,p):\mathcal{C}(-,Y)\rightarrow\mathcal{C}(-,X')$ such that $\mathcal{C}(-,f')\mathcal{C}(-,p)=\mathcal{C}(-,b)$, that is, $f'p=b$. Therefore, the functor $\alpha$ induces an equivalence $\textup{Mor}(\mathcal{C})/\mathcal{R}\cong \textup{mod-}\mathcal{C}$. We can show (b) similarly.
\end{proof}

We denote by $\mathcal{U}$  the full subcategory of $\textup{Mor}(\mathcal{C})$ consisting of $(X\xrightarrow{1} X)\oplus (Y\rightarrow 0)$ and by $\mathcal{U'}$ the full subcategory of $\textup{Mor}(\mathcal{C})$ consisting of $(X\xrightarrow{1} X)\oplus (0\rightarrow Y)$.

\begin{lem}\label{rem3.1}
Let $(a,b)$ be a morphism from $f:X\rightarrow Y$ to $f':X'\rightarrow Y'$. Then the following holds.

\textup{(a)} The morphism $b$ factors through $f'$  if and only if $(a,b)$ factors through some object in $\mathcal{U}$.

\textup{(b)} The morphism $a$ factors through $f$  if and only if $(a,b)$ factors through some object in $\mathcal{U}'$.
\end{lem}

\begin{proof}
We only prove (a). Suppose that there is a morphism $p:Y\rightarrow X'$ such that $f'p=b$, then $(a,b)$ factors through $X\oplus X'\xrightarrow{(0,1)}X'$ as $(a,b)=((a-pf,1),f')(\left(
                                                            \begin{smallmatrix}
                                                              1 \\
                                                              pf \\
                                                            \end{smallmatrix}
                                                          \right)
,p)$. Conversely, if $(a,b)$ factors through some object $A\oplus B\xrightarrow{(0,1)} B$ in $\mathcal{U}$. Assume that $(a,b)=((a_2,a_2'),b_2)(\left(
                                                            \begin{smallmatrix}
                                                              a_1 \\
                                                              a_1' \\
                                                            \end{smallmatrix}
                                                          \right)
,b_1)$, then the morphism $p=a_2'b_1:Y\rightarrow X'$ satisfies $f'p=b_2b_1=b$.
\end{proof}

Lemma \ref{prop3.1} and Lemma \ref{rem3.1} imply the following proposition.

\begin{prop}\label{prop3.1.3}
Let $\mathcal{C}$ be an additive category, then the following holds.

\textup{(a)} $\textup{Mor}(\mathcal{C})/[\mathcal{U}]\cong \textup{mod-}\mathcal{C}$.

\textup{(b)} $\textup{Mor}(\mathcal{C})/[\mathcal{U'}]\cong (\textup{mod-}\mathcal{C}^{\textup{op}})^{\textup{op}}$.
\end{prop}

\subsection{Second case: exact categories}

In this subsection we assume that $(\mathcal{C},\mathcal{E})$ is an exact category with enough projectives. We denote by $\mathcal{P}$ the full subcategory of $\mathcal{C}$ consisting of projectives. Assume that $\mathcal{M}$ is a full subcategory of $\mathcal{C}$. We denote by $\Omega\mathcal{M}$ the full subcategory of $\mathcal{C}$ formed by objects $\Omega M$ such that there is a short exact sequence $0\rightarrow \Omega M\rightarrow P\rightarrow M\rightarrow 0$ in $\mathcal{E}$ with $M\in\mathcal{M}$ and $P\in\mathcal{P}$,
by $\mathcal{M}_L$ the full subcategory of $\mathcal{C}$ consisting of objects $X$ such that there is a short exact sequence $0\rightarrow X\rightarrow M_1\rightarrow M_2\rightarrow 0$ in $\mathcal{E}$ with $M_i\in\mathcal{M}$.

For convenience, we fix some notations.  We denote by $\mathcal{U}$ the full subcategory of $\textup{Epi}(\mathcal{M})$ consisting of $(M\xrightarrow{1} M)\oplus (M'\rightarrow 0)$, by $\mathcal{V}$ the full subcategory of $\textup{Epi}(\mathcal{M})$ consisting of $(M\xrightarrow{1} M)\oplus (P\rightarrow M')$ with $P\in\mathcal{P}$, by $\mathcal{U'}$  the full subcategory of $\textup{Mono}(\mathcal{M})$ consisting of $(0\rightarrow M)\oplus (M'\xrightarrow{1} M')$  and by $\mathcal{V'}$ the subcategory of $\textup{Mono}(\mathcal{M})$ consisting of $(0\rightarrow M)\oplus (\Omega M'\rightarrow P)$ with $P\in\mathcal{P}$. We denote by Ad-Epi$(\mathcal{M})$ the full subcategory of Epi$(\mathcal{M})$ consisting of admissible epimorphisms $f:M_1\rightarrow M_2$ with $M_i\in\mathcal{M}$.

%\begin{rem}\label{rem3.0}
%In \cite{[DL]}, the authors considered the case when $(\mathcal{C},\mathcal{S})$ is an exact category with enough projectives. Actually, under the exact context, if we replace Epi$(\mathcal{M})$ with A-Epi$(\mathcal{M})$, the full subcategory of Epi$(\mathcal{M})$ consisting of admissible epimorphisms $f:M_1\rightarrow M_2$ such that $M_i\in\mathcal{M}$, then all the previous results in this subsection are still hold. The key point is the following Lemma.
%\end{rem}

\begin{defn}
A full subcategory $\mathcal{M}$ of $\mathcal{C}$ is called {\em rigid} if $\textup{Ext}^1_\mathcal{C}(M,M')=0$ for each objects $M,M'\in\mathcal{M}$.
\end{defn}

\begin{rem}\label{rem3.2}
Let $\mathcal{M}$ be a rigid subcategory of $\mathcal{C}$. If $0\rightarrow X\xrightarrow{k} M_1\xrightarrow{f} M_2\rightarrow 0$ is a short exact sequence with $M_i\in\mathcal{M}$, then $k$ is a left $\mathcal{M}$-approximation of $X$.
\end{rem}

\begin{proof}
For each $M\in\mathcal{M}$, applying $\mathcal{C}(-, M)$ to the exact sequence $0\rightarrow X\xrightarrow{k} M_1\xrightarrow{f} M_2\rightarrow 0$,
we have the following exact sequence $$0\rightarrow\mathcal{C}(M_2,M)\rightarrow\mathcal{C}(M_1,M)\rightarrow\mathcal{C}(X,M)\rightarrow \textup{Ext}^1_\mathcal{C}(M_2,M)=0.$$
Hence, $k$ is a left $\mathcal{M}$-approximation of $X$.
\end{proof}

\begin{lem}\label{lem3.2.1}
Let $\mathcal{M}$ be a full subcategory of $\mathcal{C}$ containing $\mathcal{P}$. Assume that the following diagram
$$\xymatrix{
0\ar[r]& X\ar[r]^{k}\ar[d]^{g} & M_1 \ar[r]^{f}\ar[d]^{a} & M_2\ar[r]\ar[d]^{b} & 0 \\
0\ar[r] & X'\ar[r]^{k'} & M_1' \ar[r]^{f'} &  M_2' \ar[r] & 0
}$$
is commutative with rows in $\mathcal{E}$ and $M_i, M_i'\in\mathcal{M}$. Consider the following statements:

\textup{(a)} The morphism $\underline{b}$ in $\mathcal{M}/[\mathcal{P}]$ factors through $\underline{f}'$.

\textup{(b)} The morphism $b$ factors through $f'$.

\textup{(c)} The morphism $(a,b)$ factors through some object in $\mathcal{U}$.

\textup{(d)} The morphism $g$ factors through $k$.

\textup{(e)} The morphism $g$ factors through some object in $\mathcal{M}$.

Then \textup{(a)}$\Leftrightarrow$ \textup{(b})$\Leftrightarrow$ \textup{(c)}$\Leftrightarrow$ \textup{(d)}$\Rightarrow$ \textup{(e)}. Moreover, if $\mathcal{M}$ is rigid, then all the statements are equivalent.
\end{lem}

\begin{proof}
We note that (b)$\Leftrightarrow$(c) follows from Lemma \ref{rem3.1}, (b)$\Leftrightarrow$(d) is easy, (b)$\Rightarrow$(a) and (d)$\Rightarrow$(e) are trivial. We prove (a)$\Rightarrow$(b).
Suppose that there is a morphism $\underline{p}:M_2\rightarrow M_1'$ such that $\underline{f}'\underline{p}=\underline{b}$. There exist two morphisms $u:M_2\rightarrow P$ and $v: P\rightarrow M_2'$ such that $P\in\mathcal{P}$ and $b-f'p=vu$. Since $f'$ is an admissible epimorphism and $P$ is projective, there is a morphism $w: P\rightarrow M_1'$ such that $f'w=v$. Thus, $f'(p+wu)=f'p+vu=b$.

Now assume that $\mathcal{M}$ is rigid. It remains to prove (e)$\Rightarrow$(d). Suppose that there exist two morphisms $g_1:X\rightarrow M$ and $g_2:M\rightarrow X'$ with $M\in\mathcal{M}$ such that $g=g_2g_1$. Since $k:X\rightarrow M_1$ is a left $\mathcal{M}$-approximation of $X$ by Remark \ref{rem3.2}, $g_1$ factors through $k$, thus $g$ factors through $k$.
\end{proof}

\begin{lem}\label{lem3.2.2}
Let $\mathcal{M}$ be a full subcategory of $\mathcal{C}$ containing $\mathcal{P}$. Assume that the following diagram
$$\xymatrix{
0\ar[r]& X\ar[r]^{k}\ar[d]^{g} & M_1 \ar[r]^{f}\ar[d]^{a} & M_2\ar[r]\ar[d]^{b} & 0 \\
0\ar[r] & X'\ar[r]^{k'} & M_1' \ar[r]^{f'} &  M_2' \ar[r] & 0
}$$
is commutative with rows in $\mathcal{E}$ and $M_i,M_i'\in\mathcal{M}$. Consider the following statements:

\textup{(a)} The morphism $\underline{a}$ in $\mathcal{M}/[\mathcal{P}]$ factors through $\underline{f}$.

\textup{(b)} The morphism $(a,b)$ factors through some object in $\mathcal{V}$.

\textup{(c)} The morphism $g$ factors through some object in $\Omega\mathcal{M}$.

Then \textup{(a)}$\Leftrightarrow$\textup{(b)}$\Rightarrow$\textup{(c)}. Moreover, if $\mathcal{M}$ is rigid, then all the statements are equivalent.
\end{lem}

\begin{proof}
(a)$\Rightarrow$(b). Suppose that there is a morphism $\underline{p}:M_2\rightarrow M_1'$ such that $\underline{p}\underline{f}=\underline{a}$.
Since $\mathcal{C}$ has enough projectives, there is an admissible epimorphism $a_1:P\rightarrow M_1'$ with $P\in\mathcal{P}$. Since  $a-pf$ factors through $a_1$, we assume that $a-pf=a_1a_2$ where $a_2:M_1\rightarrow P$. Now we have the following commutative diagram
$$\xymatrix{
0\ar[r]&X\ar[r]^{k}& M_1\ar[r]\ar[d]^{a_2}\ar[r]^{f} & M_2\ar@{-->}[d]^{b'} \ar[r] & 0\\
&  & P\ar[r]^{f'a_1}\ar[d]^{a_1} & M'_2 \ar@{=}[d]\ar[r] & 0 \\
0\ar[r]&X'\ar[r]^{k'}& M_1'\ar[r]^{f'} & M_2' \ar[r] & 0
}$$ with exact rows. %Since $(f'a_2)a_1=f'(a-pf)=(b-f'p)f$, there exists a morphism $b_2:Y\rightarrow M_2'$ such that $f'a_2=b_2g$ and $b-f'p=b_2b_1$ by the universal property of pushout.
Since $f'a_1a_2k=f'(a-pf)k=bfk-f'pfk=0$, there exists a morphism $b':M_2\rightarrow M_2'$ such that $b'f=f'a_1a_2$.
Since $bf=f'a=f'(a-a_1a_2)+f'a_1a_2=(f'p+b')f$ and $f$ is an epimorphism, we have $b=f'p+b'$. Thus  the following diagram
$$\xymatrix{ M_1\ar[rr]^{f}\ar[dd]^{a}\ar[rd]^{\left(
                      \begin{smallmatrix}
                        f \\
                        a_2 \\
                      \end{smallmatrix}
                    \right)
}& & M_2 \ar[dd]^{b}\ar[rd]^{\left(
                      \begin{smallmatrix}
                        1 \\
                        b' \\
                      \end{smallmatrix}
                    \right)
} &\\
& M_2\oplus P \ar[rr]^{\ \ \ \ \ \ \ \ \ \left(
                      \begin{smallmatrix}
                        1 & 0 \\
                        0 & f'a_1 \\
                      \end{smallmatrix}
                    \right)
}\ar[ld]^{(p,a_1)} & &M_2\oplus M_2'\ar[ld]^{(f'p,1)}\\
 M_1'\ar[rr]^{f'}& & M_2'&
}$$ is commutative. In other words, $(a,b)$ factors through $(M_2\oplus P\xrightarrow{\left(
                      \begin{smallmatrix}
                        1 & 0 \\
                        0 & f'a_1 \\
                      \end{smallmatrix}
                    \right)} M_2\oplus M_2')\in \mathcal{V}$.

(b)$\Rightarrow$(a). Assume hat the morphism $(a,b)$ factors through $(M\oplus P\xrightarrow{\left(
                      \begin{smallmatrix}
                        1 & 0 \\
                        0 & \pi \\
                      \end{smallmatrix}
                    \right)} M\oplus M')\in\mathcal{V}$. Suppose that the following diagram
$$\xymatrix{ M_1\ar[rr]^{f}\ar[dd]^{a}\ar[rd]^{\left(
                      \begin{smallmatrix}
                        a_1 \\
                        a'_1 \\
                      \end{smallmatrix}
                    \right)
}& & M_2 \ar[dd]^{b}\ar[rd]^{\left(
                      \begin{smallmatrix}
                        b_1 \\
                        b_1' \\
                      \end{smallmatrix}
                    \right)
} &\\
& M\oplus P \ar[rr]^{\ \ \ \ \ \ \ \ \ \left(
                      \begin{smallmatrix}
                        1 & 0 \\
                        0 & \pi \\
                      \end{smallmatrix}
                    \right)
}\ar[ld]^{(a_2,a_2')} & &M\oplus M'\ar[ld]^{(b_2,b_2')}\\
 M_1'\ar[rr]^{f'}& & M_2'&
}$$ is commutative. Let $p=a_2b_1$, then $pf=a_2b_1f=a_2a_1$, thus $\underline{a}=\underline{a_2}\underline{a_1}=\underline{p}\underline{f}$.

(b)$\Rightarrow$(c) is trivial.

Now assume that $\mathcal{M}$ is rigid. It remains to prove (c)$\Rightarrow$(a). Suppose that $g$ has a factorization $X\xrightarrow{g_1}\Omega M\xrightarrow{g_2} X'$. Then by Remark \ref{rem3.2} we complete the following commutative diagram
$$\xymatrix{
0\ar[r]& X\ar[r]^{k}\ar[d]^{g_1} & M_1 \ar[r]^{f}\ar@{-->}[d]^{a_1} & M_2\ar[r]\ar@{-->}[d]^{b_1} & 0 \\
0\ar[r]& \Omega M \ar[r]^{i}\ar[d]^{g_2} & P \ar[r]^{\pi}\ar@{-->}[d]^{a_2} & M\ar[r]\ar@{-->}[d]^{b_2} & 0\\
0\ar[r] & X'\ar[r]^{k'} & M_1' \ar[r]^{f'} &  M_2' \ar[r] & 0
}$$ with exact rows and $P\in\mathcal{P}$. Since $(a-a_2a_1)k=k'(g-g_2g_1)=0$, there exists a morphism $p:M_2\rightarrow M_1'$ such that $a-a_2a_1=pf$. Therefore, $\underline{a}=\underline{p}\underline{f}$.
\end{proof}

\begin{lem}\label{lem3.2.3}
Let $\mathcal{M}$ be a full subcategory of $\mathcal{C}$ containing $\mathcal{P}$, then

\textup{(a)} $\textup{Ad-Epi}(\mathcal{M})/[\mathcal{U}]\cong \textup{Mor}(\mathcal{M}/[\mathcal{P}])/\mathcal{R}$.

\textup{(b)} $\textup{Ad-Epi}(\mathcal{M})/[\mathcal{V}]\cong \textup{Mor} (\mathcal{M}/[\mathcal{P}])/\mathcal{R}'$.
\end{lem}

\begin{proof}
Define a functor $$\alpha: \textup{Ad-Epi}(\mathcal{M})\rightarrow\textup{Mor}(\mathcal{M}/[\mathcal{P}]),\ \ (M_1\xrightarrow{f}M_2) \mapsto (M_1\xrightarrow{\underline{f}}M_2).$$ For each object $\underline{f}:M_1\rightarrow M_2$ in $\textup{Mor}(\mathcal{M}/[\mathcal{P}])$, there is an admissible epimorphism $\pi:P\rightarrow M_2$ with $P\in\mathcal{P}$ since $\mathcal{C}$ has enough projectives. Thus $(f,\pi):M_1\oplus P\rightarrow M_2$ is an object in $\textup{Ad-Epi}(\mathcal{M})$ such that $\alpha(f,\pi)=\underline{f}$. Therefore, $\alpha$ is dense.

Assume that $f:M_1\rightarrow M_2$ and $f':M'_1\rightarrow M'_2$ are objects in $\textup{Ad-Epi}(\mathcal{M})$  and $(\underline{a},\underline{b})$ is a morphism in $\textup{Mor}(\mathcal{M}/[\mathcal{P}])$ from $\underline{f}$ to $\underline{f'}$.
Then $\underline{b}\underline{f}=\underline{f'}\underline{a}$, thus $bf-f'a$ factors through some object $P\in\mathcal{P}$. Assume that $bf-f'a=vu$ where $u:M_1\rightarrow P$ and $v:P\rightarrow M_2'$. Since $f'$ is an admissible epimorphism and $P$ is projective, there exists a morphism $w:P\rightarrow M_1'$ such that $f'w=v$. Now $(a+wu,b)$ is a morphism in $\textup{Ad-Epi}(\mathcal{M})$ from $f$ to $f'$ since $bf=f'(a+wu)$. Thus, $\alpha(a+wu,b)=(\underline{a},\underline{b})$ and the functor $\alpha$ is full.

(a) The  functor $\alpha$ induces a full and dense functor $\widetilde{\alpha}: \textup{Ad-Epi}(\mathcal{M})\rightarrow \textup{Mor}(\mathcal{M}/[\mathcal{P}])/\mathcal{R}$.
By the equivalence of (a) and (c) in Lemma \ref{lem3.2.1}, we have $\textup{Ad-Epi}(\mathcal{M})/[\mathcal{U}]\cong \textup{Mor}(\mathcal{M}/[\mathcal{P}])/\mathcal{R}$.

(b) The functor $\alpha$ induces a full and dense functor $\widehat{\alpha}: \textup{Ad-Epi}(\mathcal{M})\rightarrow \textup{Mor}(\mathcal{M}/[\mathcal{P}])/\mathcal{R}'$.
By the equivalence of (a) and (b) in Lemma \ref{lem3.2.2}, we have $\textup{Ad-Epi}(\mathcal{M})/[\mathcal{V}]\cong \textup{Mor}(\mathcal{M}/[\mathcal{P}])/\mathcal{R}'$.
\end{proof}

 Lemma \ref{lem3.2.3} and  Lemma \ref{prop3.1}  imply the following theorem, which will be crucially used in section 4 to describe the categories of short exact sequences.

\begin{thm}\label{thm3.2.1}
Let $(\mathcal{C},\mathcal{E})$ be an exact category with enough projectives. If $\mathcal{M}$ is a full subcategory of $\mathcal{C}$ containing $\mathcal{P}$,  then

\textup{(a)} $\textup{Ad-Epi}(\mathcal{M})/[\mathcal{U}]\cong \textup{mod-}(\mathcal{M}/[\mathcal{P}])$.

\textup{(b)} $\textup{Ad-Epi}(\mathcal{M})/[\mathcal{V}]\cong ((\textup{mod-}(\mathcal{M}/[\mathcal{P}])^{\textup{op}})^{\textup{op}}$.
\end{thm}

\begin{lem}\label{3.2.5}
Let $\mathcal{M}$ be a full subcategory of $\mathcal{C}$ containing $\mathcal{P}$, then $\textup{Ad-Epi}(\mathcal{M})\cong\textup{Epi}(\mathcal{M})$.
\end{lem}

\begin{proof}
We claim that the inclusion $\textup{Ad-Epi}(\mathcal{M})\hookrightarrow\textup{Epi}(\mathcal{M})$ is dense. Indeed, assume that $f:M_1\rightarrow M_2$ is an epimorphism, then we have the following commutative diagram
$$\xymatrix{
0 \ar[r] & \Omega M_2 \ar[r]\ar@{=}[d] & X \ar[r]\ar[d] & M_1\ar[r]\ar[d]^{f} & 0 \\
0 \ar[r] & \Omega M_2 \ar[r] & P \ar[r]^{\pi} & M_2\ar[r] & 0
}$$ with exact rows in $\mathcal{E}$ and $P\in\mathcal{P}$. Thus the exact sequence $0\rightarrow X\rightarrow M_1\oplus P\xrightarrow{(f,\pi)} M_2\rightarrow 0$ belongs to $\mathcal{E}$. Consequently, $(M_1\oplus P\xrightarrow{(f,\pi)} M_2)\in\textup{Ad-Epi}(\mathcal{M})$. A direct computation shows that $(M_1\xrightarrow{f} M_2)$ is isomorphic to $(M_1\oplus P\xrightarrow{(f,\pi)} M_2)$ in $\textup{Epi}(\mathcal{M})$.
\end{proof}

\begin{rem}\label{cor3.4}
Let $(\mathcal{C},\mathcal{E})$ be an exact category with enough projectives, then $\mathcal{C}(\mathcal{E})$, $\textup{Ad-Epi}(\mathcal{C})$ and $\textup{Epi}(\mathcal{C})$ are equivalent.
\end{rem}

%\begin{proof}
%It is well known that the category $\mathcal{C}(\mathcal{E})$ and $\textup{Ad-Epi}(\mathcal{C})$ are equivalent. The corollary follows from Lemma \ref{3.2.5}.
%\end{proof}
\begin{cor}
Let  $\mathcal{M}$ be a full subcategory of $\mathcal{C}$ containing $\mathcal{P}$. If $\mathcal{M}$ is closed under kernel of epimorphisms, denote by $\textup{Mono}_C(\mathcal{M})$ the full subcategory of $\textup{Mono}(\mathcal{M})$ consisting of monomorphisms $f$ such that $\textup{Coker}(f)\in\mathcal{M}$, then

\textup{(a)} $\textup{Mono}_C(\mathcal{M})/[\mathcal{U'}]\cong\textup{mod-}(\mathcal{M}/[\mathcal{P}])$.

\textup{(b)} $\textup{Mono}_C(\mathcal{M})/[\mathcal{V'}]\cong(\textup{mod-}(\mathcal{M}/[\mathcal{P}])^{\textup{op}})^{\textup{op}}$.
\end{cor}

\begin{proof}
By assumption, the kernel functor $\textup{Ker}: \textup{Epi}(\mathcal{M})\rightarrow \textup{Mono}_C(\mathcal{M})$ induces two equivalences
$$\textup{Epi}(\mathcal{M})/[\mathcal{U}]\cong \textup{Mono}_C(\mathcal{M})/[\mathcal{U'}], \ \textup{Epi}(\mathcal{M})/[\mathcal{V}]\cong \textup{Mono}_C(\mathcal{M})/[\mathcal{V'}].$$
The corollary follows from Theorem \ref{thm3.2.1} and Lemma \ref{3.2.5}. We can compare (a) with \cite[Theorem 3.3]{[Haf]}.
\end{proof}

\begin{lem}\label{lem3.2.4}
Let  $\mathcal{M}$ be a full and rigid subcategory of $\mathcal{C}$ containing $\mathcal{P}$.  Then

\textup{(a)} $\mathcal{M}_L/[\mathcal{M}]\cong \textup{Ad-Epi}(\mathcal{M})/[\mathcal{U}]$.

\textup{(b)} $\mathcal{M}_L/[\Omega\mathcal{M}]\cong \textup{Ad-Epi}(\mathcal{M})/[\mathcal{V}]$.
\end{lem}

\begin{proof}
Define a functor by $$\beta: \textup{Ad-Epi}(\mathcal{M})\rightarrow\mathcal{M}_L, \ \ (M_1\xrightarrow{f} M_2) \mapsto \textup{Ker}(f).$$
Then $\beta$ is dense. For each morphism $g:X\rightarrow X'$ in $\mathcal{M}_L$, there exists the following diagram
$$\xymatrix{
0\ar[r]& X\ar[r]^{k}\ar[d]^{g} & M_1 \ar[r]^{f} \ar@{-->}[d]^{a} & M_2\ar[r]\ar@{-->}[d]^{b} & 0 \\
0\ar[r] & X'\ar[r]^{k'} & M_1' \ar[r]^{f'} &  M_2' \ar[r] & 0
}$$
 with rows in $\mathcal{E}$ and $M_i,M_i'\in\mathcal{M}$. Since $\mathcal{M}$ is rigid, by Remark \ref{rem3.2} there exists a morphism $a:M_1\rightarrow M_1'$ such that $ak=k'g$.
Then there is a morphism $b: M_2\rightarrow M_2'$ such that $bf=f'a$. Hence $\beta(a,b)=g$ and the functor $\beta$ is full.

(a) The functor $\beta$ induces a full and dense functor $\widetilde{\beta}: \textup{Ad-Epi}(\mathcal{M})\rightarrow\mathcal{M}_L/[\mathcal{M}]$. We note that $\widetilde{\beta}(\mathcal{U})=0$. %Assume that $(a,b)$ is a morphism in $\textup{Epi}(\mathcal{M})$ and $g:X\rightarrow X'$ is a morphism in $\mathcal{M}_L$ such that $\widetilde{\beta}(a,b)=\underline{g}=0$. Then there exist two morphisms $g_1:X\rightarrow M$ and $g_2:M\rightarrow X'$ with $M\in\mathcal{M}$ such that $g=g_2g_1$. Since $k:X\rightarrow _1$ is a left $\mathcal{M}$-approximation of $\mathcal{M}$, $g_1$ factors through $k$, thus $g$ factors through $k$. It follows that $b$ factors through $f'$, that is, $(a,b)$ factors through some object in $\mathcal{U}$. Consequently, $\textup{Epi}(\mathcal{M})/[\mathcal{U}]\cong\mathcal{M}_L/[\mathcal{M}]$.
The equivalence $\textup{Ad-Epi}(\mathcal{M})/[\mathcal{U}]\cong\mathcal{M}_L/[\mathcal{M}]$ follows from the equivalent statements of (c) and (e) in Lemma \ref{lem3.2.1}. %Therefore, $\mathcal{M}_L/[\mathcal{M}]\cong \textup{Epi}(\mathcal{M})/[\mathcal{U}]$ by Lemma \ref{3.2.5}.

(b) The functor $\beta$ induces a full and dense functor $\widehat{\beta}: \textup{Ad-Epi}(\mathcal{M})\rightarrow\mathcal{M}_L/[\Omega\mathcal{M}]$. Since $\widehat{\beta}(\mathcal{V})=0$,  the equivalence $\textup{Ad-Epi}(\mathcal{M})/[\mathcal{V}]\cong\mathcal{M}_L/[\Omega\mathcal{M}]$ follows from the equivalent statements of (b) and (c) in Lemma \ref{lem3.2.2}. %Therefore, $\mathcal{M}_L/[\Omega\mathcal{M}]\cong \textup{Epi}(\mathcal{M})/[\mathcal{V}]$ by Lemma \ref{3.2.5}.
\end{proof}

By Lemma \ref{lem3.2.4} and Theorem \ref{thm3.2.1}, we have the following corollary, where (a) was appeared in \cite[Theorem 3.2]{[DL]}.% and (b) is dual to \cite[Theorem 3.5]{[DL]}.

\begin{cor}\label{cor3.1}
Let $(\mathcal{C},\mathcal{E})$ be an exact category with enough projectives. If $\mathcal{M}$ is a rigid and full subcategory of $\mathcal{C}$ containing $\mathcal{P}$, then

\textup{(a)} $\mathcal{M}_L/[\mathcal{M}]\cong \textup{mod-}(\mathcal{M}/[\mathcal{P}])$.

\textup{(b)} $\mathcal{M}_L/[\Omega\mathcal{M}]\cong (\textup{mod-}(\mathcal{M}/[\mathcal{P}])^{\textup{op}})^{\textup{op}}$.
\end{cor}

Suppose that $\mathcal{M}$ is a contravariantly finite subcategory of $\mathcal{C}$ containing $\mathcal{P}$, then by Proposition \ref{prop3.1.3}, $\textup{Mor}(\mathcal{M})/[\mathcal{U}]\cong \textup{mod-}\mathcal{M}$ is abelian. Moreover, $\textup{Epi}(\mathcal{M})/[\mathcal{U}]\cong\textup{mod-}\mathcal{M}/[\mathcal{P}]$ is abelian by Theorem \ref{thm3.2.1}. The following result is a variant of \cite[Theorem 3.7]{[AHK]}.

\begin{prop}
Let $\mathcal{C}$ be an abelian category with enough projectives. If $\mathcal{M}$ is a contravariantly finite subcategory containing all projectives, then there exists a recollement
$$  \xymatrix @C=2.7pc {
      \textup{Epi}(\mathcal{M})/[\mathcal{U}] \ar[r]^{i_*} & \textup{Mor}(\mathcal{M})/[\mathcal{U}]   \ar@<+2.8ex>[l]_{i^!}\ar@<-2.8ex>[l]_{i^*}\ar[r]^{\ \ \ \ \ \ j^*}
      & \mathcal{C} \ar@<+2.8ex>[l]_{\ \ \ \ \ \ j_*}\ar@<-2.8ex>[l]_{\ \ \ \ \ \ j_!}
      } $$
of abelian categories.
\end{prop}

\begin{proof}
Consider the cokernel functor
$$\textup{Cok}:\textup{Mor}(\mathcal{M})\rightarrow\mathcal{C},\ \ \ (M_1\xrightarrow{f} M_2)\mapsto \textup{Coker}(f).$$
Since $\textup{Cok}(X\xrightarrow{1} X)=0$ and $\textup{Cok}(X\rightarrow 0)=0$ for each $X\in \mathcal{M}$, the functor $\textup{Cok}$ induces a functor $j^\ast: \textup{Mor}(\mathcal{M})/[\mathcal{U}]\rightarrow\mathcal{C}$. For each $X\in\mathcal{C}$, there exists an exact sequence $P_1\xrightarrow{g} P_0\rightarrow X\rightarrow 0$ with $P_i\in\mathcal{P}$ since $\mathcal{C}$ has enough projectives. It is easy to check that the functor $$j_!:\mathcal{C}\rightarrow \textup{Mor}(\mathcal{M})/[\mathcal{U}],\ \ \  X\mapsto (P_1\xrightarrow{g} P_0)$$
is well defined. For each $X\in \mathcal{C}$, there exist two left $\mathcal{M}$-approximations $a:M_1\rightarrow X$ and $b:M_2\rightarrow \textup{Ker}(a)$ with $M_i\in\mathcal{M}$. Since $\mathcal{M}$ contains all projectives, $a$ and $b$ are epimorphisms. Thus we have an exact sequence $M_2\xrightarrow{h} M_1\xrightarrow{a} X\rightarrow 0$ where $h$ is the composition of $b$ and  the natural inclusion Ker$(a)\hookrightarrow M_1$. Define a functor by
$$j_\ast:\mathcal{C}\rightarrow \textup{Mor}(\mathcal{M})/[\mathcal{U}],\ \ \  X\mapsto (M_2\xrightarrow{h} M_1).$$
It is routine to prove that $(j_\ast,j^\ast)$ and $(j^\ast,j_!)$ are adjoint pairs. Moreover, $j_\ast$ and $j_!$ are fully-faithful. We note that $\textup{Ker}(j^\ast)=\textup{Epi}(\mathcal{M})/[\mathcal{U}]$, so by \cite[Remark 2.3]{[Ps]} we complete the proof. Actually, the functors $i^\ast,i_\ast$ and $i^!$ are described as follows:
$$i^\ast:\textup{Mor}(\mathcal{M})/[\mathcal{U}]\rightarrow \textup{Epi}(\mathcal{M})/[\mathcal{U}], \ \ \ (M_1\xrightarrow{f} M_2)\mapsto (M_1\oplus P\xrightarrow{(f,\pi)} M_2)$$
$$i_\ast:\textup{Epi}(\mathcal{M})/[\mathcal{U}]\rightarrow \textup{Mor}(\mathcal{M})/[\mathcal{U}], \ \ \ (M_1\xrightarrow{f} M_2)\mapsto (M_1\xrightarrow{f} M_2)$$
$$i^!:\textup{Mor}(\mathcal{M})/[\mathcal{U}]\rightarrow \textup{Epi}(\mathcal{M})/[\mathcal{U}], \ \ \ (M_1\xrightarrow{f} M_2)\mapsto (M_1\xrightarrow{f} \textup{Im}(f))$$
where $\pi:P\rightarrow M_2$ is an epimorphism with $P\in\mathcal{P}$.
\end{proof}

\begin{cor}\label{cor3.2}
 Let $\mathcal{C}$ be an abelian category with enough projectives.
Then there exists a recollement
$$  \xymatrix @C=2.7pc {
      \textup{mod-}(\mathcal{C}/[\mathcal{P}]) \ar[r]^{\ \ \ i_*} & \textup{mod-}\mathcal{C}  \ar@<+2.8ex>[l]_{\ \ \ i^!}\ar@<-2.8ex>[l]_{\ \  \ i^*}\ar[r]^{\ \ \ \ \ j^*}
      & \mathcal{C} \ar@<+2.8ex>[l]_{\ \ \ \ \ j_*}\ar@<-2.8ex>[l]_{ \ \ \ \ \  j_!}
      } $$
of abelian categories. Therefore, we have an equivalence $\textup{mod-}\mathcal{C}/[\textup{mod-}\mathcal{C}/[\mathcal{P}]]\cong \mathcal{C}$.
\end{cor}

\begin{rem}
Following Lenzing \cite{[L]}, the equivalence $\textup{mod-}\mathcal{C}/[\textup{mod-}\mathcal{C}/[\mathcal{P}]]\cong \mathcal{C}$ is called {\em Auslander's formula}; see \cite{[Aus]}.
\end{rem}

\subsection{Third case: triangulated categories}

Let $\mathcal{C}$ be a right triangulated category with suspension functor $\Sigma$ and $\mathcal{M}$ be a full subcategory of $\mathcal{C}$. We denote by $\mathcal{M}\ast\Sigma\mathcal{M}$ the full subcategory of $\mathcal{C}$ consisting of objects $X$ such that there is a right triangle $M_1\rightarrow M_2\rightarrow X\rightarrow\Sigma M_1$ with $M_i\in\mathcal{M}$. A full subcategory $\mathcal{M}$ is called {\em rigid} if $\mathcal{C}(M,\Sigma M')=0$ for each $M,M'\in\mathcal{M}$.

\begin{lem}\label{3.3.1}(\cite[Lemma 1.3]{[ABM]}) Let $X_{1}\xrightarrow{f_1} X_2\xrightarrow{f_2} X_3\xrightarrow{f_3}\Sigma X_1$ be a right triangle, then the following holds.

\textup{(a)} $f_{i+1}$ is a weak cokernel of $f_i$ for $i=1,2$.

\textup{(b)} If $\Sigma$ is fully-faithful, then $f_i$ is a weak kernel of $f_{i+1}$ for $i=1,2$.
\end{lem}

The following result generalizes \cite[Proposition 6.2]{[IY]} from triangulated categories to right triangulated categories.

\begin{prop} \label{prop3.0}
Let $\mathcal{C}$ be a right triangulated category and $\mathcal{M}$ be a rigid subcategory. If $\Sigma$ is fully-faithful, then $(\mathcal{M}\ast \Sigma\mathcal{M})/[\Sigma\mathcal{M}]\cong \textup{mod-}\mathcal{M}$.
\end{prop}

\begin{proof}
By Lemma \ref{prop3.1}, we only need to show that $(\mathcal{M}\ast \Sigma\mathcal{M})/[\Sigma\mathcal{M}]\cong \textup{Mor}(\mathcal{M})/\mathcal{R}$. For each morphism $f_1:M_1\rightarrow M_2$ in $\mathcal{M}$, we assume that $M_1\xrightarrow{f_1} M_2\xrightarrow{f_2} X\xrightarrow{f_3}\Sigma M_1$ is a right triangle. The assignment $(M_1\xrightarrow{f_1} M_2)\mapsto X$ defines a dense functor $F:\textup{Mor}(\mathcal{M})\rightarrow\mathcal{M}\ast \Sigma\mathcal{M}$. Assume that $g: X\rightarrow X'$ is a morphism in $\mathcal{M}\ast \Sigma\mathcal{M}$. Since $\mathcal{M}$ is rigid, $f_3'gf_2=0$. Thus by Lemma \ref{3.3.1}, there exists a morphism $b:M_2\rightarrow M_1$ such that $gf_2=f_2'b$. Since $f_2'bf_1=0$, there is a morphism $a: M_1\rightarrow M_2$ such that $bf_1=f'_1a$ by Lemma \ref{3.3.1}.  Hence, $F(a,b)=g$ and the functor $F$ is full.
$$\xymatrix{
M_1 \ar[r]^{f_1}\ar@{-->}[d]^{a} & M_2 \ar[r]^{f_2}\ar@{-->}[d]^{b} & X \ar[r]^{f_3}\ar[d]^{g} & \Sigma M_1\ar@{-->}[d]^{\Sigma a}\\ M_1' \ar[r]^{f_1'} & M_2' \ar[r]^{f_2'} & X' \ar[r]^{f_3'} & \Sigma M_1'
}$$ We note that $F$ induces a full and dense functor $\widetilde{F}:\textup{Mor}(\mathcal{M})\rightarrow \mathcal{M}\ast \Sigma\mathcal{M}/[\Sigma\mathcal{M}]$. As in the above diagram, we assume that $\widetilde{F}(a,b)=\underline{g}=0$. Since  $\mathcal{M}$ is rigid,  $f_3$ is a left $\Sigma\mathcal{M}$-approximation, thus $g$ factors through $f_3$. Hence, $f_2'b=gf_2=0$ and $b$ factors through $f_1'$, which implies that $\textup{Mor}(\mathcal{M})/\mathcal{R}\cong(\mathcal{M}\ast \Sigma\mathcal{M})/[\Sigma\mathcal{M}]$. We complete the proof.
\end{proof}

We recall that a full subcategory $\mathcal{M}$ of a triangulated category $\mathcal{C}$ is {\em cluster-tilting} if $\mathcal{M}$ is rigid and $\mathcal{C}=\mathcal{M}\ast\Sigma\mathcal{M}$.

\begin{cor} (\cite{[BMR],[KR],[KZ]}) \label{cor3.3}
Let $\mathcal{C}$ be a triangulated category with suspension functor $\Sigma$ and $\mathcal{M}$ be a cluster-tilting subcategory, then there is an equivalence of categories $\mathcal{C}/[\Sigma \mathcal{M}]\cong \textup{mod-}\mathcal{M}$.
\end{cor}

\subsection{Examples}

Let $A$ be an Artin $k$-algebra and $ \mathcal{C}$ be a full subcategory of mod-$A$ containing $A_A$. Denote by $\mathcal{U}$ the full subcategory of $\textup{Epi}(\mathcal{C})$ consisting of $(X\xrightarrow{1} X)\oplus (Y\rightarrow 0)$. Assume that all the indecomposable objects in $\mathcal{C}$ are $M_1,M_2,\cdots, M_n$. Set $M=\oplus_{i=1}^nM_i$. Then $B=\textup{End}_A(M)$ is called {\em Auslander algebra} of $\mathcal{C}$ and $\underline{B}=\underline{\textup{End}}_A(M)$ is called {\em stable Auslander algebra} of $\mathcal{C}$. It is easy to see that  $\underline{B}=B/BeB$, where $e$ is the idempotent given by Hom$_A(M,A)$.

\begin{prop} With notations as above. Then

\textup{(a)} $\textup{Mor}(\mathcal{C})/[\mathcal{U}]\cong\textup{mod-}B$.

\textup{(b)} $\textup{Epi}(\mathcal{C})/[\mathcal{U}]\cong\textup{mod-}\underline{B}$.

\end{prop}

\begin{proof}
(a) Since $\mathcal{C}=\textup{add}M$ and the functor $\textup{Hom}_A(M,-)$ induces an equivalence $\textup{add-}M\cong \textup{proj-}B$,  we have the following equivalence
$$\alpha: \textup{Mor}(\mathcal{C})/[\mathcal{U}]=\textup{Mor}(\mathcal{C})/\mathcal{R}\cong\textup{Mor}(\textup{proj-}B)/\mathcal{R}\cong\textup{mod-}B$$
which mapping a morphism $ f: X\rightarrow Y$ to CokerHom$_A(M,f)$.

(b)   Assume that $f: X\rightarrow Y$ is a morphism in $\mathcal{C}$, then
$$\begin{array}{rl}
      & f\ \textup{is\ an\ epimorphism} \\
    \Leftrightarrow &\textup{Hom}_B(\textup{Hom}_A(M,A),\textup{Coker}\textup{Hom}_A(M,f))=0 \\
     \Leftrightarrow &\textup{Hom}_B(eB, \alpha(f))=0 \\
    \Leftrightarrow & \alpha(f)e=0 \\
    \Leftrightarrow & \alpha(f)\in\textup{mod-}B/BeB=\textup{mod-}\underline{B}\\
  \end{array}$$
where the first if and only if condition follows from \cite[Section 6]{[RZ]}. Thus, the functor $\alpha$ induces an equivalence $\textup{Epi}(\mathcal{C})/[\mathcal{U}]\cong\textup{mod-}\underline{B}$.
\end{proof}

\begin{example}
Let $A$ be a representation-finite Artin $k$-algebra. Then $$\textup{Epi}(\textup{mod-}A)/[\mathcal{U}]\cong\textup{mod-}\underline{B}$$
 where $\underline{B}$ is the stable Auslander algebra of $A$. In the case when $A$ is self-injective, this equivalence was proved in \cite[Theorem 1]{[Eir]} using the language of submodule category.
\end{example}

\begin{example}
Let $A$ be an Artin algebra of CM-finite type. By definition, an algebra $A$ is called {\em CM-finite type} if the number of indecomposable Gorenstein projective modules up to isomorphisms is finite. Denote by  $\textup{Gproj-}A$ the full subcategory of $\textup{mod-}A$ formed by Gorenstein projective modules. Then  $$\textup{Epi}(\textup{Gproj-}A)/[\mathcal{U}]\cong\textup{mod-}\underline{B}$$ where $\underline{B}$ is the sable Auslander Cohen-Macaulay algebra of $A$.
\end{example}

%\begin{defn}
%Let $\mathcal{C}$ be an $n$-angulated category. A full subcategory $\mathcal{M}$ of $\mathcal{C}$ is called \em{cluster-tilting} if $\mathcal{M}$ is rigid and for each object $X\in\mathcal{C}$ there exists an $n$-angle
%$$M_1\rightarrow M_2\rightarrow\cdots\rightarrow M_{n-1}\rightarrow X\rightarrow \Sigma M_1$$
%with all $M_i\in\mathcal{M}$.
%\end{defn}

%\begin{prop}
%Let $\mathcal{C}$ be an $n$-angulated category and $\mathcal{M}$ be a cluster-tilting subcategory, then there is an equivalence of categories $\mathcal{C}/[\Sigma \mathcal{M}]\simeq \textup{mod-}\mathcal{M}$.
%\end{prop}

\section{Abelian quotients of the categories of short exact sequences}

In this section, we assume that $(\mathcal{C},\mathcal{E})$ is an exact category.
We always view a short exact sequence $0\rightarrow X_1\xrightarrow{f_1} X_2\xrightarrow{f_2} X_3\rightarrow 0$, sometimes $(X_1\rightarrow X_2\rightarrow X_3)$ for short, as a complex $X_\bullet$ concentrated on degree 1,2 and 3. When we say $\varphi_\bullet: X_\bullet\rightarrow Y_\bullet$ is a morphism between two short exact sequences $X_\bullet$ and $Y_\bullet$, we means that the following diagram
$$\xymatrix{ 0\ar[r] & X_1 \ar[r]^{f_1}\ar[d]^{\varphi_1} & X_2 \ar[r]^{f_2}\ar[d]^{\varphi_2} & X_3 \ar[r]\ar[d]^{\varphi_3} & 0\\  0\ar[r] & Y_1 \ar[r]^{g_1} & Y_2 \ar[r]^{g_2} & Y_3 \ar[r] & 0\\
}$$ is commutative. We denote by $C^b(\mathcal{C})$ the category of bounded complexes over $\mathcal{C}$, by $\mathcal{E}(\mathcal{C})$ the full subcategory of $C^b(\mathcal{C})$ consisting of short exact sequences in $\mathcal{E}$, and by $S\mathcal{E}(\mathcal{C})$ the full subcategory of $\mathcal{E}(\mathcal{C})$ formed by split short exact sequences over $\mathcal{C}$.

Throughout this section, if $(\mathcal{C},\mathcal{E})$ has enough projectives, we always denote by $\mathcal{P}$ the full subcategory of $\mathcal{C}$ formed by all projectives. Similarly, if $(\mathcal{C},\mathcal{E})$ has enough injectives, we always denote by $\mathcal{I}$ the full subcategory of $\mathcal{C}$ formed by all injectives.

\subsection{Realizing quotients of the categories of short exact sequences as module categories}

%\begin{proof}
%The implication (a) $\Rightarrow$ (b) is clear. For (b) $\Rightarrow$ (a), since $(\varphi_2-g_1p_1)f_1=0$, there exists a morphism $p_2:X_3\rightarrow Y_2$ such that $\varphi_2-g_1p_1=p_2f_2$. That is $\varphi_2=g_1p_1+p_2f_2$. Since $\varphi_3f_2=g_2\varphi_2=g_2(g_1p_1+p_2f_2)=g_2p_2f_2$ and $f_2$ is epic, we obtain that $\varphi_3=g_2p_2$. Thus $\varphi_\bullet$ is homotopic to zero. We can prove (a) $\Leftrightarrow$ (c) dually.
%\end{proof}

%By Lemma \ref{lem1} and Lemma \ref{3.2.5}, we have the following.

%\begin{cor}\label{cor4.1}
%Let $(\mathcal{C},\mathcal{E})$ be an exact category. Then
%$$\mathcal{E}(\mathcal{C})/[S\mathcal{E}(\mathcal{C})]\simeq \textup{Epi}(\mathcal{C})/R\simeq \textup{Mono}(\mathcal{C})/R'.$$
%\end{cor}

%The following theorem immediately follows from Corollary \ref{cor4.1} and Theorem \ref{thm3.2.1} together with its dual.

\begin{thm}\label{thm4.1}
Let $(\mathcal{C},\mathcal{E})$ be an exact category and $X_\bullet:0\rightarrow X_1\xrightarrow{f_1} X_2\xrightarrow{f_2} X_3\rightarrow 0$ be a short exact sequence in $\mathcal{E}$.

\textup{(a)} If $(\mathcal{C},\mathcal{E})$ has enough projectives, %denote by $\mathcal{P}$ the full subcategory of $\mathcal{C}$ formed by all projectives.
 then we have the following equivalences:
$$\alpha_1:\mathcal{E}(\mathcal{C})/[S\mathcal{E(}\mathcal{C})]\cong \textup{mod-}\mathcal{C}/[\mathcal{P}],\ \ X_\bullet \mapsto \textup{Coker}(\mathcal{C}/[\mathcal{P}](-,f_2))$$
$$\alpha_2:\mathcal{E}(\mathcal{C})/[P\mathcal{E}(\mathcal{C})]\cong (\textup{mod-}(\mathcal{C}/[\mathcal{P}])^{\textup{op}})^{\textup{op}}, \ \ X_\bullet \mapsto \textup{Coker}(\mathcal{C}/[\mathcal{P}](f_2,-))$$ where $P\mathcal{E}(\mathcal{C})$ is the full subcategory of $\mathcal{E}(\mathcal{C})$ formed by $(0\rightarrow X\rightarrow X)\oplus(\Omega Y\rightarrow P\rightarrow Y)$.

\textup{(b)} If $(\mathcal{C},\mathcal{E})$ has enough injectives, %denote by $\mathcal{I}$ the full subcategory of $\mathcal{C}$ formed by all injectives.
then we have the following equivalences:
$$\beta_1:\mathcal{E}(\mathcal{C})/[S\mathcal{E}(\mathcal{C})]\cong (\textup{mod-}(\mathcal{C}/[\mathcal{I}])^{\textup{op}})^{\textup{op}}, \ \ X_\bullet\mapsto \textup{Coker}(\mathcal{C}/[\mathcal{I}](f_1,-))$$
$$\beta_2:\mathcal{E}(\mathcal{C})/[I\mathcal{E}(\mathcal{C})]\cong \textup{mod-}\mathcal{C}/[\mathcal{I}], \ \ X_\bullet\mapsto\textup{Coker}(\mathcal{C}/[\mathcal{I}](-,f_1))$$ where $I\mathcal{E}(\mathcal{C})$ is the full subcategory of $\mathcal{E}(\mathcal{C})$ formed by $(X\rightarrow X\rightarrow 0)\oplus( Y\rightarrow I\rightarrow \Omega^{-1}Y)$.
\end{thm}

\begin{proof}
We only prove (a).  We have two equivalences $\mathcal{E}(\mathcal{C})/[S\mathcal{E(}\mathcal{C})]\cong \textup{Ad-Epi}(\mathcal{C})/[\mathcal{U}]$ and $\mathcal{E}(\mathcal{C})/[P\mathcal{E}(\mathcal{C})]\cong \textup{Ad-Epi}(\mathcal{C})/[\mathcal{V}]$, where $\mathcal{U}$ (resp. $\mathcal{V}$) is the full subcategory of $\textup{Ad-Epi}(\mathcal{C})$ consisting of $(X\xrightarrow{1} X)\oplus (Y\rightarrow 0)$ (resp. $(X\xrightarrow{1} X)\oplus (P\rightarrow Y)$ with $P\in\mathcal{P}$). Then (a) follows from Theorem  \ref{thm3.2.1}.
\end{proof}

\begin{cor}\label{cor4.2}
Let $(\mathcal{C},\mathcal{E})$ be a Frobenius  category, then all the quotient categories $\mathcal{E}(\mathcal{C})/[S\mathcal{E}(\mathcal{C})]$, $\mathcal{E}(\mathcal{C})/[P\mathcal{E}(\mathcal{C})]$  and $\mathcal{E}(\mathcal{C})/[I\mathcal{E}(\mathcal{C})]$ are equivalent to abelian category $\textup{mod-}\mathcal{C}/[\mathcal{P}]$.
\end{cor}

\begin{rem} \label{rem4.1}
Let $(\mathcal{C},\mathcal{E})$ be a Frobenius category. Suppose that $X_\bullet:0\rightarrow X_1\rightarrow X_2\rightarrow X_3\rightarrow 0$ is a short exact sequence in $\mathcal{E}$, then we have the following commutative diagram
 $$\xymatrix{
 0\ar[r] & \Omega X_3\ar[r]\ar[d] & P\ar[r]\ar[d] & X_3 \ar[r]\ar@{=}[d] & 0\\
 0\ar[r] & X_1\ar[r]\ar@{=}[d] & X_2\ar[r]\ar[d] & X_3 \ar[r]\ar[d] & 0\\
 0\ar[r] & X_1\ar[r] & I\ar[r] & \Omega^{-1}X_1 \ar[r] & 0}$$ with rows in $\mathcal{E}$. Hence, the equivalences between the quotient categories in Corollary \ref{cor4.2} can be described as the following rotations:
$$\begin{array}{ll}
    \alpha=\alpha_2^{-1}\beta_1: & \mathcal{E}(\mathcal{C})/[S\mathcal{E}(\mathcal{C})]\cong \mathcal{E}(\mathcal{C})/[P\mathcal{E}(\mathcal{C})], \\
     & X_\bullet \mapsto (\Omega X_3\rightarrow P\oplus X_1\rightarrow X_2)
  \end{array}
$$$$\begin{array}{ll}
    \beta=\beta_2^{-1}\alpha_1: & \mathcal{E}(\mathcal{C})/[S\mathcal{E}(\mathcal{C})]\cong \mathcal{E}(\mathcal{C})/[I\mathcal{E}(\mathcal{C})]. \\
     & X_\bullet \mapsto ( X_2\rightarrow X_3\oplus I\rightarrow \Omega^{-1}X_1)
  \end{array}
$$
\end{rem}

By the dual of Lemma \ref{3.2.5}, the following is an equivalent statement of Corollary \ref{cor4.2}. We can compare it with \cite[Theorem 1]{[Eir]}.

\begin{cor}
Let $(\mathcal{C},\mathcal{E})$ be a Frobenius category. Denote by $\mathcal{P}$ the full subcategory of projective-injective objects in $\mathcal{C}$, by $\mathcal{U}_1$ the full subcategory of $\textup{Mono}(\mathcal{C})$ consisting of $(X\xrightarrow{1} X)\oplus (0\rightarrow Y)$, by $\mathcal{U}_2$ the full subcategory of $\textup{Mono}(\mathcal{C})$ consisting of $(X\xrightarrow{1} X)\oplus (Y\rightarrow P)$ with $P\in\mathcal{P}$ and by $\mathcal{U}_3$ the full subcategory of $\textup{Mono}(\mathcal{C})$ consisting of $(0\rightarrow X)\oplus (Y\rightarrow P)$ with $P\in\mathcal{P}$. Then all the quotient categories $\textup{Mono}(\mathcal{C})/[\mathcal{U}_1]$, $\textup{Mono}(\mathcal{C})/[\mathcal{U}_2]$  and $\textup{Mono}(\mathcal{C})/[\mathcal{U}_3]$ are equivalent to $\textup{mod-}\mathcal{C}/[\mathcal{P}]$.
\end{cor}

%\begin{example}
%Let $A$ be an Artin algebra and $M$ be an $n$-cluster tilting object in $\textup{mod-}A$. Then $\mathcal{C}=\textup{add}M$ is an $n$-abelian category. Thus we have $$E_n(\mathcal{C})/[CE_n(\mathcal{C})]\simeq\textup{mod-}\underline{B},$$ where $\underline{B}$ is the stable $n$-Auslander algebra of $\mathcal{C}$.
%\end{example}

\subsection{Abelian structure}
Let $(\mathcal{C},\mathcal{E})$ be an exact category with enough projectives. Then by Theorem \ref{thm4.1}, $\mathcal{E}(\mathcal{C})/[S\mathcal{E}(\mathcal{C})]$ is equivalent to $\textup{mod-}\mathcal{C}/[\mathcal{P}]$ thus  has an abelian structure.  In this subsection, we will prove that for general exact category $(\mathcal{C},\mathcal{E})$, the quotient category $\mathcal{E}(\mathcal{C})/[S\mathcal{E}(\mathcal{C})]$ always has an abelian structure given by pushout and pullback diagrams. %Moreover, the two abelian structures are the same when $\mathcal{C}$ has enough projectives.

\begin{lem}\label{lem0}
Assume that the following diagram
$$\xymatrix{X_\bullet\ar[d]^{\varphi_\bullet} & 0\ar[r] & X_1 \ar[r]^{f_1}\ar[d]^{\varphi_1} & X_2 \ar[r]^{f_2}\ar[d]^{\varphi_2} & X_3 \ar[r]\ar[d]^{\varphi_3} & 0\\
Y_\bullet & 0\ar[r] & Y_1 \ar[r]^{g_1} & Y_2 \ar[r]^{g_2} & Y_3 \ar[r] & 0\\
}$$
is commutative with rows in $\mathcal{E}$. Then the following diagram
$$\begin{gathered}\xymatrix{
K(\varphi_\bullet)\ar[d]^{k_\bullet} & 0\ar[r] & X_1 \ar[r]^{\left(
                                                                                        \begin{smallmatrix}
                                                                                          f_1  \\
                                                                                           \varphi_1\\
                                                                                        \end{smallmatrix}
                                                                                      \right)}\ar@{=}[d] & X_2\oplus Y_1 \ar[r]^{\left(
                                                                                        \begin{smallmatrix}
                                                                                          a_1 & -h_1 \\
                                                                                        \end{smallmatrix}
                                                                                      \right)}\ar[d]^{(1,0)
 } & Z \ar[r]\ar[d]^{h_2} & 0\\
X_\bullet\ar[d]^{\pi_\bullet} & 0\ar[r] & X_1 \ar[r]^{f_1}\ar[d]^{\varphi_1\ \ \  (I)} & X_2 \ar[r]^{f_2}\ar[d]^{a_1} & X_3 \ar[r]\ar@{=}[d] & 0\\
I(\varphi_\bullet)\ar[d]^{i_\bullet} & 0\ar[r] & Y_1 \ar[r]^{h_1}\ar@{=}[d] & Z \ar[r]^{h_2}\ar[d]^{a_2\ \ \ \  (II)} & X_3 \ar[r]\ar[d]^{\varphi_3} & 0 \\
Y_\bullet\ar[d]^{c_\bullet} & 0\ar[r] & Y_1 \ar[r]^{g_1}\ar[d]^{h_1} & Y_2 \ar[r]^{g_2}\ar[d]^{\left(
                                                                                                                            \begin{smallmatrix}
                                                                                                                              0 \\
                                                                                                                              1 \\
                                                                                                                            \end{smallmatrix}
                                                                                                                          \right)} & Y_3 \ar[r]\ar@{=}[d] & 0\\
C(\varphi_\bullet) & 0\ar[r] &  Z \ar[r]^{\left(
                                                                                        \begin{smallmatrix}
                                                                                          h_2  \\
                                                                                          a_2\\
                                                                                        \end{smallmatrix}
                                                                                      \right)} & X_3\oplus Y_2 \ar[r]^{(-\varphi_3, g_2)} & Y_3 \ar[r] & 0\\
}\end{gathered}\eqno{(4.1)}$$
is commutative with rows in $\mathcal{E}$, moreover,  $\varphi_\bullet=i_\bullet\pi_\bullet$.
\end{lem}

\begin{proof}
%Let $Z$ be the pushout of $f_1$ and $\varphi_1$. Since $g_1\varphi_1=\varphi_2f_1$, there exists a  unique morphism $a_2:Z\rightarrow Y_2$ such that $g_1=a_2h_1$ and $\varphi_2=a_2a_1$. Then $g_2a_2=\varphi_3h_2$ follows from the universal property. The sequences $K(\varphi_\bullet)$ and $C_(\varphi_\bullet)$ are exact since the square (I) and (II) are both pushout and pullback diagrams.
By \cite[Proposition 3.1]{[B]}, the morphism $\varphi_\bullet$ factors through some short exact sequence $I(\varphi_\bullet)$ in $\mathcal{E}$ in such a way that $\varphi_\bullet=i_\bullet\pi_\bullet$ and the squares (I) and (II) are both pushout and pullback diagrams. The sequences $K(\varphi_\bullet)$ and $C(\varphi_\bullet)$ belong to $\mathcal{E}$ since the squares (I) and (II) are both pushout and pullback diagrams.
\end{proof}

\begin{lem}\label{lem1}(\cite[Proposition 1.1]{[Gen]})
Let $\varphi_\bullet: X_\bullet\rightarrow Y_\bullet$ be a morphism in $\mathcal{E}(\mathcal{C})$. Then the following statements are equivalent.

\textup{(a)} There is a morphism $p_1:X_2\rightarrow Y_1$ such that $\varphi_1=p_1f_1$.

\textup{(b)} There is a morphism $p_2:X_3\rightarrow Y_2$ such that $\varphi_3=g_2p_2$.

\textup{(c)} The morphism $\varphi_\bullet$ is homotopic to zero.

\textup{(d)} The morphism $\varphi_\bullet$ factors through a split short exact sequence.

\textup{(e)} The morphism $\underline{\varphi_\bullet}=0$ in $\mathcal{E}(\mathcal{C})/[S\mathcal{E}(\mathcal{C})]$.
\end{lem}

\begin{lem}\label{lem4.2}
 Let $\varphi_\bullet:X_\bullet\rightarrow Y_\bullet$ be a morphism in $\mathcal{E}(\mathcal{C})$.  Then $\underline{\varphi_\bullet}$ is a monomorphism in $\mathcal{E}(\mathcal{C})/[S\mathcal{E}(\mathcal{C})]$ if and only if
$\left(
                                                                                        \begin{smallmatrix}
                                                                                          f_1 \\
                                                                                          \varphi_1 \\
                                                                                        \end{smallmatrix}
                                                                                      \right)$ is a section.
\end{lem}

\begin{proof}
For the ``if" part, assume that there exists a morphism $(f_1',\varphi_1'): X_2\oplus Y_1 \rightarrow X_1$ such that $(f_1',\varphi_1')\left(
                                                                                        \begin{smallmatrix}
                                                                                          f_1 \\
                                                                                          \varphi_1 \\
                                                                                        \end{smallmatrix}
                                                                                      \right)=1$. Suppose that $\psi_\bullet: Z_\bullet\rightarrow X_\bullet$ is
a morphism such that $\underline{\varphi_\bullet\psi_\bullet}=0$.
$$\xymatrix{
Z_\bullet\ar[d]^{\psi_\bullet} & 0\ar[r] & Z_1 \ar[r]^{h_1}\ar[d]^{\psi_1} & Z_2 \ar[r]^{h_2}\ar[d]^{\psi_2} & Z_3 \ar[r]\ar[d]^{\psi_3} & 0\\
X_\bullet\ar[d]^{\varphi_\bullet} & 0\ar[r] & X_1 \ar[r]^{f_1}\ar[d]^{\varphi_1} & X_2 \ar[r]^{f_2}\ar[d]^{\varphi_2} & X_3 \ar[r]\ar[d]^{\varphi_3} & 0\\
Y_\bullet & 0\ar[r] & Y_1 \ar[r]^{g_1} & Y_2 \ar[r]^{g_2} & Y_3 \ar[r] & 0\\
}$$
By Lemma \ref{lem1}, there is a morphism $p_1:Z_2\rightarrow Y_1$ such that $\varphi_1\psi_1=p_1h_1$. Thus there exists a morphism $q_1=(f_1',\varphi'_1)\left(
                                 \begin{smallmatrix}
                                   \psi_2 \\
                                   p_1 \\
                                 \end{smallmatrix}
                               \right):Z_2\rightarrow X_1$ such that
$q_1h_1=(f_1'\psi_2+\varphi'_1p_1)h_1=(f_1'f_1+\varphi'_1\varphi_1)\psi_1=\psi_1$. We infer that $\underline{\psi_\bullet}=0$ by Lemma \ref{lem1} again.

For the ``only if" part, there is a morphism $p_1=(0,1):X_2\oplus Y_1\rightarrow Y_1$ such that $\varphi_1=p_1\left(
                                                                                        \begin{smallmatrix}
                                                                                          f_1 \\
                                                                                          \varphi_1 \\
                                                                                        \end{smallmatrix}
 \right)$, so we have $\underline{\varphi_\bullet k_\bullet}=\underline{i_\bullet\pi_\bullet k_\bullet}=0$ by Lemma \ref{lem0} and Lemma \ref{lem1}. Since $\underline{\varphi_\bullet}$ is a monomorphism, we have $\underline{k_\bullet}=0$, thus $\left(
                                                                                        \begin{smallmatrix}
                                                                                          f_1 \\
                                                                                          \varphi_1 \\
                                                                                        \end{smallmatrix}
                                                                                      \right)$ is a section by Lemma \ref{lem1}. %Since $f_1$ is not a section, $\varphi_1$ is a section.
\end{proof}

When $\mathcal{C}$ is a certain abelian category, the following theorem was appeared in \cite[Theorem 2.5]{[Gen1]}.

\begin{thm}\label{thm4.3.1}
Let $(\mathcal{C},\mathcal{E})$ be an exact category. Then the quotient $\mathcal{E}(\mathcal{C})/[S\mathcal{E}(\mathcal{C})]$ is an abelian category whose kernels and cokernels are given by pullback and pushout diagrams.
\end{thm}

\begin{proof} %For convenience, we give an alternative proof.
Suppose that $\varphi_\bullet: X_\bullet\rightarrow Y_\bullet$ is a morphism in $\mathcal{E}(\mathcal{C})$. As notations in diagram (4.1), we claim that $\underline{k_\bullet}: K(\varphi_\bullet)\rightarrow X_\bullet$ is a kernel of $\underline{\varphi_\bullet}$.
By Lemma \ref{lem4.2}, $\underline{k_\bullet}$ is a monomorphism. Since $\varphi_1=(0\ 1)\left(
                                                                                        \begin{smallmatrix}
                                                                                          f_1 \\
                                                                                          \varphi_1 \\
                                                                                        \end{smallmatrix}
                                                                                      \right)$, it follows from Lemma \ref{lem1} that $\underline{\varphi_\bullet k_\bullet}=0$.
Assume that there is a morphism $\psi_\bullet: Z_\bullet\rightarrow X_\bullet$ such that $\underline{\varphi_\bullet\psi_\bullet}=0$, then by Lemma \ref{lem1}, there is a morphism $p_1:Z_2\rightarrow Y_1$ such that $\varphi_1\psi_1=p_1h_1$. Since $\left(
                                                                                        \begin{smallmatrix}
                                                                                          \psi_2  \\
                                                                                           p_1\\
                                                                                        \end{smallmatrix}
                                                                                      \right)h_1=\left(
                                                                                        \begin{smallmatrix}
                                                                                          f_1  \\
                                                                                           \varphi_1\\
                                                                                        \end{smallmatrix}
                                                                                      \right)\psi_1$, we obtain the following commutative diagram:
$$\xymatrix{
\psi_\bullet \ar[d]^{\theta_\bullet} & 0 \ar[r] & Z_1 \ar[r]^{h_1} \ar[d]^{\psi_1} & Z_2 \ar[r]^{h_2}\ar[d]^{\left(
                                                                                        \begin{smallmatrix}
                                                                                          \psi_2  \\
                                                                                           p_1\\
                                                                                        \end{smallmatrix}
                                                                                      \right)} & Z_3 \ar[r] \ar[d]^{\theta_3} & 0\\
K(\varphi_\bullet) & 0\ar[r] & X_1 \ar[r]^{\left(
                                                                                        \begin{smallmatrix}
                                                                                          f_1  \\
                                                                                           \varphi_1\\
                                                                                        \end{smallmatrix}
                                                                                      \right)} & X_2\oplus Y_1 \ar[r]^{\left(
                                                                                        \begin{smallmatrix}
                                                                                          a_1 & -h_1 \\
                                                                                        \end{smallmatrix}
                                                                                      \right)}
 & Z \ar[r] & 0\\
 }$$
By a direct checking, we have  $\underline{\psi_\bullet}=\underline{k_\bullet}\underline{\theta_\bullet}$.

Dually we can show that $\underline{c_\bullet}:Y_\bullet\rightarrow C(\varphi_\bullet)$ is a cokernel of $\underline{\varphi_\bullet}$.

It remains to show that $\textup{Coker}(\textup{Ker}(\underline{\varphi_\bullet}))\cong \textup{Ker}(\textup{Coker}(\underline{\varphi_\bullet}))$, that is,
$\textup{Coker}(\underline{k_\bullet})\cong \textup{Ker}(\underline{c_\bullet})$. Indeed, the following  commutative diagram
$$\xymatrix{
\textup{Coker}(\underline{k_\bullet})\ar[d] & 0 \ar[r] & X_2\oplus Y_1 \ar[r]^{\left(
                                          \begin{smallmatrix}
                                            1 & 0 \\
                                            a_1 & -h_1 \\
                                          \end{smallmatrix}
                                        \right)} \ar@{=}[d] & X_2\oplus Z \ar[r]^{\left(
                                          \begin{smallmatrix}
                                            f_2 & -h_2 \\
                                          \end{smallmatrix}
                                        \right)}\ar[d]^{\left(
                                          \begin{smallmatrix}
                                           1 & 0 \\
                                           a_1 & -1 \\
                                          \end{smallmatrix}
                                        \right)} & X_3 \ar[r]\ar@{=}[d] & 0\\
I(\varphi_\bullet) & 0 \ar[r] & X_2\oplus Y_1 \ar[r]^{\left(
                                          \begin{smallmatrix}
                                           1 & 0 \\
                                            0 & h_1\\
                                          \end{smallmatrix}
                                        \right)}  & X_2\oplus Z \ar[r]^{\left(
                                          \begin{smallmatrix}
                                           0 & h_2 \\
                                          \end{smallmatrix}
                                        \right)} & X_3 \ar[r] & 0\\
}$$ shows that $\textup{Coker}(\underline{k_\bullet})\cong I(\varphi_\bullet)$. The following commutative diagram
$$\xymatrix{
I(\varphi_\bullet)\ar[d] & 0 \ar[r] & Y_1 \ar[r]^{\left(
                                          \begin{smallmatrix}
                                           h_1\\
                                            0 \\
                                          \end{smallmatrix}
                                        \right)}\ar@{=}[d]  & Z\oplus Y_2 \ar[r]^{\left(
                                          \begin{smallmatrix}
                                            h_2 & 0 \\
                                            0 & 1 \\
                                          \end{smallmatrix}
                                        \right)}\ar[d]^{\left(
                                          \begin{smallmatrix}
                                           1 & 0 \\
                                            a_2 & -1\\
                                          \end{smallmatrix}
                                        \right)} & X_3\oplus Y_2 \ar[r]\ar@{=}[d] & 0\\
\textup{Ker}(\underline{c_\bullet}) & 0 \ar[r] &  Y_1 \ar[r]^{\left(
                                          \begin{smallmatrix}
                                           h_1 \\
                                           g_1 \\
                                          \end{smallmatrix}
                                        \right)}  & Z\oplus Y_2 \ar[r]^{\left(
                                          \begin{smallmatrix}
                                            h_2 & 0 \\
                                           a_2 & -1\\
                                          \end{smallmatrix}
                                        \right)} & X_3\oplus Y_2 \ar[r] & 0\\
}$$
implies that $I(\varphi_\bullet)\cong\textup{Ker}(\underline{c_\bullet})$. We are done.
\end{proof}

\begin{rem}
Let $(\mathcal{C},\mathcal{E})$ be an exact category with enough projectives. Then  $\mathcal{E}(\mathcal{C})/[S\mathcal{E}(\mathcal{C})]\cong \textup{mod-}\mathcal{C}/[\mathcal{P}]$ thus has an abelian structure.  Theorem \ref{thm4.3.1} tells us that the quotient $\mathcal{E}(\mathcal{C})/[S\mathcal{E}(\mathcal{C})]$ has an abelian structure given by pushout and pullback diagrams. In fact, the two abelian structures are the same by Remark \ref{rem2.2}.
\end{rem}

\begin{rem}\label{rem1}
Let $\underline{\varphi_\bullet}: X_\bullet\rightarrow Y_\bullet$ be a monomorphism in $\mathcal{E}(\mathcal{C})/[S\mathcal{E}(\mathcal{C})]$. Then by Lemma \ref{lem0}, we have $\underline{\varphi_\bullet}=\underline{i_\bullet\pi_\bullet}$. Note that $\underline{\pi_\bullet}:X_\bullet\rightarrow I(\varphi_\bullet)$ is both a monomorphism and an epimorphism, thus it is an isomorphism. Therefore, for convenience when we mention a monomorphism $\underline{\varphi_\bullet}: X_\bullet\rightarrow Y_\bullet$ in $\mathcal{E}(\mathcal{C})/[S\mathcal{E}(\mathcal{C})]$, we can assume that $\varphi_1=1$.
\end{rem}

\subsection{ Projective objects, injective objects and Hilton-Rees Theorem}

 Let $(\mathcal{C},\mathcal{E})$ be an exact category. We first provide the projective objects and injective objects in $\mathcal{E}(\mathcal{C})/[S\mathcal{E}(\mathcal{C})]$.

\begin{prop}\label{prop1}
Let $(\mathcal{C},\mathcal{E})$ be an exact category.

\textup{(a)} Each short exact sequence $P_X: 0\rightarrow \Omega X\xrightarrow{f_1} P\xrightarrow{f_2} X\rightarrow 0$  in $\mathcal{E}$  with $P$ projective is a projective  object in $\mathcal{E}(\mathcal{C})/[S\mathcal{E}(\mathcal{C})]$.

\textup{(b)} If $(\mathcal{C},\mathcal{E})$ has enough projectives, then each projective object in $\mathcal{E}(\mathcal{C})/[S\mathcal{E}(\mathcal{C})]$ is of the form $P_X$ for some object $X$ in $\mathcal{C}$. In this case, $\mathcal{E}(\mathcal{C})/[S\mathcal{E}(\mathcal{C})]$ has enough projectives.
\end{prop}

%\begin{lem}\label{rem4.2.1}
%Each short exact sequence $P_X: 0\rightarrow \Omega X\xrightarrow{f_1} P\xrightarrow{f_2} X\rightarrow 0$ with $P$ projective is a projective  object in $\mathcal{E}(\mathcal{C})/[S\mathcal{E}(\mathcal{C})]$. Dually, each short exact sequence $I_X: 0\rightarrow X\xrightarrow{f_1} I\xrightarrow{f_2} \Omega^{-1}X\rightarrow 0$ with $I$ injective is an injective  object in $\mathcal{E}(\mathcal{C})/[S\mathcal{E}(\mathcal{C})]$.
%\end{lem}

\begin{proof}
(a) Assume that $\underline{\varphi_\bullet}:Y_\bullet\rightarrow Z_\bullet$ is an epimorphism and  $\underline{\psi_\bullet}:P_X\rightarrow Z_\bullet$  is a morphism in $\mathcal{E}(\mathcal{C})/[S\mathcal{E}(\mathcal{C})]$. By the dual version of Remark \ref{rem1}, we assume that $\varphi_3=1$. Since $P$ is projective, we obtain a morphism $\underline{\phi_\bullet}:P_X\rightarrow Y_\bullet$ such that $\phi_3=\psi_3$. Since $\psi_3=\varphi_3\phi_3$, it follows that $\underline{\psi_\bullet}=\underline{\varphi_\bullet\phi_\bullet}$ by Lemma \ref{lem1}.
 Therefore, $P_X$ is projective.

(b) Suppose that $X_\bullet: 0\rightarrow X_1\xrightarrow{f_1} X_2\xrightarrow{f_2} X_3\rightarrow 0$  is an object in $\mathcal{E}(\mathcal{C})$. Since $\mathcal{C}$ has enough projectives, there exists a short exact sequence $P_{X_3}: 0\rightarrow \Omega X_3\xrightarrow{g_1} P\xrightarrow{g_2} X_3\rightarrow 0$ in $\mathcal{E}$ with $P$ projective.
Thus we have the following commutative diagram:
$$\xymatrix{P_{X_3}\ar[d]^{\varphi_\bullet} & 0\ar[r] & \Omega X_3 \ar[r]^{g_1}\ar@{-->}[d]^{\varphi_1} & P \ar[r]^{g_2}\ar@{-->}[d]^{\varphi_2} & X_3 \ar[r]\ar@{=}[d] & 0\\
X_\bullet & 0\ar[r] & X_1 \ar[r]^{f_1} & X_2 \ar[r]^{f_2} & X_3 \ar[r] & 0\\
}$$ Consequently, $\underline{\varphi_\bullet}:P_{X_3}\rightarrow X_\bullet$ is an epimorphism where $P_{X_3}$ is projective by (a). In particular, assume that $X_\bullet$ is a projective object in $\mathcal{E}(\mathcal{C})/[S\mathcal{E}(\mathcal{C})]$, then there is an epimorphism $\underline{\varphi_\bullet}:P_{X_3}\rightarrow X_\bullet$. Since $X_\bullet$ is projective, $\underline{\varphi_\bullet}$ is split. Thus each projective object of $\mathcal{E}(\mathcal{C})/[S\mathcal{E}(\mathcal{C})]$ is of the form $P_{X}$ for some object $X$ in $\mathcal{C}$.
\end{proof}

\begin{cor}
Let $(\mathcal{C},\mathcal{E})$ be a Frobenius category, then  $\mathcal{E}(\mathcal{C})/[S\mathcal{E}(\mathcal{C})]$ is a Frobenius abelian category.
\end{cor}

\begin{rem}
Assume that $(\mathcal{C},\mathcal{E})$ is an exact category with enough projectives. If $\mathcal{C}$ admits an additive generator $M$, then $P_M$ is a projective generator for $\mathcal{E}(\mathcal{C})/[S\mathcal{E}(\mathcal{C})]$. Therefore, $\mathcal{E}(\mathcal{C})/[S\mathcal{E}(\mathcal{C})]\cong \textup{mod-End}P_M\cong \textup{mod-}\underline{B}$, where $\underline{B}$ is the stable Auslander algebra of $\mathcal{C}$. See subsection 3.4 for more details.
\end{rem}

Recall that given a short exact sequence $\delta:  0\rightarrow X_1\xrightarrow{f_1} X_2\xrightarrow{f_2} X_3\rightarrow 0$ in $\mathcal{E}$, we define  the {\em contravariant defect} $\delta^\ast$  and the {\em covariant defect} $\delta_\ast$ by the following exact sequence of functors
$$0\rightarrow\mathcal{C}(-,X_{1})\xrightarrow{\mathcal{C}(-,f_{1})}\mathcal{C}(-,X_{2})\xrightarrow{\mathcal{C}(-,f_{2})}\mathcal{C}(-,X_3)\rightarrow \delta^\ast\rightarrow 0,$$
$$0\rightarrow\mathcal{C}(X_{3},-)\xrightarrow{\mathcal{C}(f_{2},-)}\mathcal{C}(X_{2},-)\xrightarrow{\mathcal{C}(f_{1},-)}\mathcal{C}(X_1,-)\rightarrow \delta_\ast\rightarrow 0.$$

\begin{example}\label{ex4.1}
(a) Let $\delta=P_X: 0\rightarrow \Omega X\rightarrow P\rightarrow X\rightarrow 0$ with $P\in\mathcal{P}$. Then  $\delta^\ast=\mathcal{C}/[\mathcal{P}](-,X)$ and $\delta_\ast=\textup{Ext}^1_\mathcal{C}(X,-)$.

(b) Let $\delta=I_X: 0\rightarrow X\rightarrow I\rightarrow \Omega^{-1}X\rightarrow 0$ with $I\in\mathcal{I}$. Then $\delta^\ast=\textup{Ext}^1_\mathcal{C}(-,X)$ and $\delta_\ast=\mathcal{C}/[\mathcal{I}](X,-)$.
\end{example}

\begin{rem}\label{rem4.2}
Let $(\mathcal{C},\mathcal{E})$ be an exact category with enough projectives and injectives.

 (a) In Theorem \ref{thm4.1}, the equivalence  $\alpha_1:\mathcal{E}(\mathcal{C})/[S\mathcal{E}(\mathcal{C})]\cong \textup{mod-}\mathcal{C}/[\mathcal{P}]$ is given by $\delta\mapsto \delta^\ast$, and the equivalence $\beta_1:\mathcal{E}(\mathcal{C})/[S\mathcal{E}(\mathcal{C})]\cong (\textup{mod-}(\mathcal{C}/[\mathcal{I}])^{\textup{op}})^{\textup{op}}$ is given by $\delta\mapsto \delta_\ast$.

(b) In $\textup{mod-}\mathcal{C}/[\mathcal{P}]$, each projective object is of the form $\mathcal{C}/[\mathcal{P}](-,X)$, and each injective object is of the form $\textup{Ext}^1_\mathcal{C}(-,X)$.
\end{rem}

\begin{proof}
(a) Assume that $\delta:  0\rightarrow X_1\xrightarrow{f_1} X_2\xrightarrow{f_2} X_3\rightarrow 0$ is a short exact sequence in $\mathcal{E}$. Recall that $\alpha_1(\delta)=\textup{Coker}(\mathcal{C}/[\mathcal{P}](-,f_2))$ and $\delta^\ast=\textup{Coker}(\mathcal{C}(-,f_2))$. Since $\delta^\ast(\mathcal{P})=0$, we can view $\delta^\ast$ as a finitely presented $\mathcal{C}/[\mathcal{P}]$-module by Proposition 2.4 (b). Thus $\alpha_1(\delta)=\delta^\ast$. Similarly, we have $\beta_1(\delta)=\delta_\ast$.

(b) It follows from  Proposition \ref{prop1} and Example \ref{ex4.1} since $\mathcal{E}(\mathcal{C})/[S\mathcal{E}(\mathcal{C})]\cong \textup{mod-}\mathcal{C}/[\mathcal{P}]$.
\end{proof}

\begin{prop}\label{prop4.3.1}
Let $(\mathcal{C},\mathcal{E})$ be an exact category with enough projectives and injectives. Then there is a duality
$$\Phi:\textup{mod-}\mathcal{C}/[\mathcal{P}]\rightarrow\textup{mod-}(\mathcal{C}/[\mathcal{I}])^{\textup{op}},\ \ \ \delta^\ast\mapsto \delta_\ast.$$
Moreover, by restrictions, we obtain the following two dualities
$$\Phi:\textup{proj-}\mathcal{C}/[\mathcal{P}]\rightarrow\textup{inj-}(\mathcal{C}/[\mathcal{I}])^{\textup{op}}, \ \ \ \mathcal{C}/[\mathcal{P}](-,X)\mapsto \textup{Ext}^1_\mathcal{C}(X,-).$$
$$\Phi:\textup{inj-}\mathcal{C}/[\mathcal{P}]\rightarrow\textup{proj-}(\mathcal{C}/[\mathcal{I}])^{\textup{op}}, \ \ \ \textup{Ext}^1_\mathcal{C}(-,X)\mapsto\mathcal{C}/[\mathcal{I}](X,-).$$
\end{prop}

\begin{proof}
It is a direct consequence of Remark \ref{rem4.2}(a) and Example \ref{ex4.1}.
\end{proof}

The following result is implied in Proposition \ref{prop4.3.1}.

\begin{thm} (Hilton-Rees Theorem, see \cite{[HR],[Ma]})
Let $(\mathcal{C},\mathcal{E})$ be an exact category with enough projectives and injectives.

\textup{(a)} There is an isomorphism between $\mathcal{C}/[\mathcal{P}](Y,X)$ and the group of natural transformations from $\textup{Ext}^1_\mathcal{C}(X,-)$ to $\textup{Ext}^1_\mathcal{C}(Y,-)$.

\textup{(b)} There is an isomorphism between $\mathcal{C}/[\mathcal{I}](X,Y)$ and the group of natural transformations from $\textup{Ext}^1_\mathcal{C}(-,X)$ to $\textup{Ext}^1_\mathcal{C}(-,Y)$.
\end{thm}

The following is a variant of \cite[Section 7]{[AR]}.

\begin{prop}
Let $(\mathcal{C},\mathcal{E})$ be an exact category with enough projectives and $F$ be an object in $\textup{mod-}\mathcal{C}/[\mathcal{P}]$.
Then there exists a short exact sequence $0\rightarrow X_1\rightarrow X_2\rightarrow X_3\rightarrow 0$ in $\mathcal{E}$, such that the following sequence
$$\cdots\rightarrow \mathcal{C}/[\mathcal{P}](-,\Omega^2 X_3)\rightarrow \mathcal{C}/[\mathcal{P}](-,\Omega X_1)\rightarrow \mathcal{C}/[\mathcal{P}](-,\Omega X_2)\rightarrow\mathcal{C}/[\mathcal{P}](-,\Omega X_3)\rightarrow$$
$$\begin{gathered}\rightarrow\mathcal{C}/[\mathcal{P}](-,X_1)\rightarrow\mathcal{C}/[\mathcal{P}](-,X_2)\rightarrow\mathcal{C}/[\mathcal{P}](-,X_3)\rightarrow F\rightarrow 0 \end{gathered}\eqno{(4.2)}$$
is a projective resolution of $F$.
Moreover, if $(\mathcal{C},\mathcal{E})$ has enough injectives, then the following sequence
$$0\rightarrow F\rightarrow\textup{Ext}^1_\mathcal{C}(-,X_1)\rightarrow\textup{Ext}^1_\mathcal{C}(-,X_2)\rightarrow
\textup{Ext}^1_\mathcal{C}(-,X_3)\rightarrow\textup{Ext}^2_\mathcal{C}(-,X_1)\rightarrow$$
$$\begin{gathered}\rightarrow\textup{Ext}^2_\mathcal{C}(-,X_2)\rightarrow\textup{Ext}^2_\mathcal{C}(-,X_3)
\rightarrow\textup{Ext}^3_\mathcal{C}(-,X_1)\rightarrow\cdots \end{gathered}\eqno{(4.3)}$$
is an injective resolution of $F$.
\end{prop}

\begin{proof}
The existence of short exact sequence $0\rightarrow X_1\rightarrow X_2\rightarrow X_3\rightarrow 0$ follows from the equivalence $\mathcal{E}(\mathcal{C})/[S\mathcal{E}(\mathcal{C})]\cong\textup{mod-}\mathcal{C}/[\mathcal{P}]$. Thus we have an exact sequence
$$0\rightarrow\mathcal{C}(-,X_1)\rightarrow\mathcal{C}(-,X_2)\rightarrow \mathcal{C}(-,X_3)\rightarrow F\rightarrow 0.$$
A direct checking proves that the sequence $$\begin{gathered}\mathcal{C}/[\mathcal{P}](-,X_1)\rightarrow\mathcal{C}/[\mathcal{P}](-,X_2)\rightarrow \mathcal{C}/[\mathcal{P}](-,X_3)\rightarrow F\rightarrow 0 \end{gathered}\eqno{(4.4)}$$ is exact.
The following commutative diagram
$$\xymatrix{
0\ar[r] & \Omega X_3 \ar[r]\ar[d] & P \ar[r]\ar[d] & X_3\ar[r]\ar@{=}[d] & 0\\
0\ar[r] & X_1\ar[r] & X_2 \ar[r] & X_3 \ar[r] & 0
}$$ implies that $\Omega X_3\rightarrow X_1\rightarrow X_2\rightarrow X_3$ is a left triangle in $\mathcal{C}/[\mathcal{P}]$.
Since $\mathcal{C}/[\mathcal{P}](X,-)$ is a homological functor, the sequence
$$\cdots\rightarrow \mathcal{C}/[\mathcal{P}](-,\Omega X_2)\rightarrow\mathcal{C}/[\mathcal{P}](-,\Omega X_3)
\rightarrow\mathcal{C}/[\mathcal{P}](-,X_1)\rightarrow$$
$$\begin{gathered}\rightarrow\mathcal{C}/[\mathcal{P}](-,X_2)\rightarrow
\mathcal{C}/[\mathcal{P}](-,X_3)\end{gathered}\eqno{(4.5)}$$ is exact. The sequences (4.4) and (4.5) together show that sequence (4.2) is a projective resolution of $F$.

Since the sequence (4.3) is exact, it remains to show that $\textup{Ext}^i_\mathcal{C}(-,X)$ is injective for $i\geq 2$. Indeed, since $(\mathcal{C},\mathcal{E})$ has enough injectives, %there is a  short exact sequence
%$0\rightarrow X\rightarrow I\rightarrow \Omega^{-1}X\rightarrow 0$ with $I$ injective.
by choosing injective envelopes, we have $\textup{Ext}^i_\mathcal{C}(-,X)\cong\textup{Ext}^{i-1}_\mathcal{C}(-,\Omega^{-1}X)\cong\cdots\cong
\textup{Ext}^1_\mathcal{C}(-,\Omega^{-i+1}X).$
\end{proof}

\subsection{Simple objects and Auslander-Reiten theory}

In this subsection, we always assume that $(\mathcal{C},\mathcal{E})$ is an Ext-finite $k$-linear exact category, where Ext-finite means that all morphism and extension modules $\mathcal{C}(X,Y)$ and Ext$^i_\mathcal{C}(X,Y)$ have finite length over $k$.

Recall that a non-split exact sequence $0\rightarrow X_1\xrightarrow{f_1} X_2\xrightarrow{f_2} X_3\rightarrow 0$ is called an {\em Auslander-Reiten sequence} if the following two conditions are satisfied:

(a) If $g:X_1\rightarrow Y$ is not a section, then $g$ factors through $f_1$.

(b) If $h:Z\rightarrow X_3$ is not a retraction, then $h$ factors through $f_2$.

We say  $\mathcal{C}$ has {\em right }(resp. {\em left}){\em Auslander-Reiten sequences} if  each non-projective (resp. non-injective) object is the ending (resp. starting) term of  an Auslander-Reiten sequence. We say $\mathcal{C}$ has {\em  Auslander-Reiten sequences} if it has both right and left Auslander-Reiten sequences.

\begin{lem}\label{lem4.4.1}
 Let $X_\bullet$ be a simple object in $\mathcal{E}(\mathcal{C})/[S\mathcal{E}(\mathcal{C})]$, then $X_\bullet$ is isomorphic to $X'_\bullet: 0\rightarrow X'_1\xrightarrow{f'_1} X'_2\xrightarrow{f'_2} X'_3\rightarrow 0$, where $X_1'$ and $X_3'$ are indecomposable.
\end{lem}

\begin{proof}
Assume that $X_\bullet$ is of the form $0\rightarrow X_1\xrightarrow{f_1} X_2\xrightarrow{f_2} X_3\rightarrow 0$ with $f_1, f_2\in J_\mathcal{C}$. Suppose that $X_1=X_1'\oplus X_1''$, where $X_1'$ is indecomposable. Then there exist two canonical morphisms $i:X_1'\rightarrow X_1$ and $\pi:X_1\rightarrow X_1'$ such that $\pi i=1$. Considering the pushout of $f_1$ and $\pi$, we have the following commutative diagram
$$\xymatrix{X_\bullet\ar[d]^{\varphi_\bullet} & 0\ar[r] & X_1 \ar[r]^{f_1}\ar[d]^{\pi} & X_2 \ar[r]^{f_2}\ar[d]^{\varphi_2} & X_3 \ar[r]\ar@{=}[d] & 0\\
X'_\bullet & 0\ar[r] & X'_1 \ar[r]^{g_1} & X'_2 \ar[r]^{g_2} & X_3 \ar[r] & 0\\
}$$
whose second row belongs to $\mathcal{E}$. Since $X_\bullet$ is a simple object, $\underline{\varphi_\bullet}$ is either zero or a monomorphism. Noting $\underline{\varphi_\bullet}$ is an epimorphism, we claim that $\underline{\varphi_\bullet}$ is  a monomorphism thus is an isomorphism. Otherwise, $\underline{\varphi_\bullet}=0$, thus there exists a morphism $p:X_2\rightarrow X_1'$ such that $\pi=pf_1$. The following commutative diagram
$$\xymatrix{
 0\ar[r] & X'_1 \ar[r]^{1}\ar[d]^{i} & X'_1 \ar[r]\ar[d]^{f_1i} & 0 \ar[r]\ar[d] & 0\\
0\ar[r] & X_1 \ar[r]^{f_1}\ar[d]^{\pi} & X_2 \ar[r]^{f_2}\ar[d]^{p} & X_3 \ar[r]\ar[d] & 0\\
 0\ar[r] & X'_1 \ar[r]^{1} & X'_1 \ar[r]  & 0 \ar[r] & 0\\
}$$ implies that $f_1\notin J_\mathcal{C}$. It is a contradiction. Therefore, $X_\bullet$ is isomorphic to $X_\bullet'$.
\end{proof}

We compare the following result with \cite[Propostion 14]{[Rump]} and \cite[Proposition 4.1]{[BJ]}.

\begin{thm}\label{thm4.4.1}
Let $(\mathcal{C},\mathcal{E})$ be an Ext-finite $k$-linear exact category.

\textup{(a)} Assume that $X_\bullet: 0\rightarrow X_1\xrightarrow{f_1} X_2\xrightarrow{f_2} X_3\rightarrow 0$ is a non-split short exact sequence in $\mathcal{E}$ where $X_1$ and $X_3$ are indecomposable. Then $X_\bullet$ is a simple object in $\mathcal{E}(\mathcal{C})/[S\mathcal{E}(\mathcal{C})]$ if and only if $X_\bullet$ is an Auslander-Reiten sequence in $\mathcal{C}$.

\textup{(b)} There is a bijection between the set of isoclasses of simple objects in $\mathcal{E}(\mathcal{C})/[S\mathcal{E}(\mathcal{C})]$ and the set of isoclasses of Auslander-Reiten sequences in $\mathcal{C}$.
\end{thm}

\begin{proof}
(a) For the ``only if" part, suppose that $\varphi_1:X_1\rightarrow Y_1$ is not a section, then we have the following commutative diagram
$$\xymatrix{X_\bullet\ar[d]^{\varphi_\bullet} & 0\ar[r] & X_1 \ar[r]^{f_1}\ar[d]^{\varphi_1} & X_2 \ar[r]^{f_2}\ar[d]^{\varphi_2} & X_3 \ar[r]\ar@{=}[d] & 0\\
Y_\bullet & 0\ar[r] & Y_1 \ar[r]^{g_1} & Y_2 \ar[r]^{g_2} & X_3 \ar[r] & 0\\
}$$
which is induced by the pushout of $f_1$ and $\varphi_1$. Since $f_1$ and $\varphi_1$ are not sections and $X_1$ is indecomposable, the morphism $\left(
                                                                                                                      \begin{smallmatrix}
                                                                                                                        f_1 \\
                                                                                                                        \varphi_1 \\
                                                                                                                      \end{smallmatrix}
                                                                                                                    \right)$
is not a section. Thus $\underline{\varphi_\bullet}$ is not a monomorphism by Lemma \ref{lem4.2}. We infer that $\underline{\varphi_\bullet}=0$ since $X_\bullet$ is a simple object.  It follow from Lemma \ref{lem1} that $\varphi_1$ factors through $f_1$. Dually, we can prove that if $\varphi_3: Z_3\rightarrow X_3$ is not a retraction, then $\varphi_3$ factors through $f_2$. Thus $X_\bullet$ is an Auslander-Reiten sequence.

For the `` if" part, assume that $\varphi_\bullet: X_\bullet\rightarrow Y_\bullet$ is a morphism in $\mathcal{E}$. If $\varphi_1$ is a section, then $\left(
                                                                                                                      \begin{smallmatrix}
                                                                                                                        f_1 \\
                                                                                                                        \varphi_1 \\
                                                                                                                      \end{smallmatrix}
                                                                                                                    \right)$ is a section, thus $\underline{\varphi_\bullet}$ is a monomorphism. If $\varphi_1$ is not a section, then $\varphi_1$ factors through $f_1$ since $X_\bullet$ is an Auslander-Reiten sequence, thus $\underline{\varphi_\bullet}=0$. Therefore, each morphism $\underline{\varphi_\bullet}: X_\bullet\rightarrow Y_\bullet$ is either a monomorphism or a zero morphism. It means that $X_\bullet$ is a simple object in $\mathcal{E}(\mathcal{C})/[S\mathcal{E}(\mathcal{C})]$.

(b) It follows from Lemma \ref{lem4.4.1} and (a).
\end{proof}

From now on to the end of this subsection, we assume that $(\mathcal{C},\mathcal{E})$ is an exact category with enough projectives and  injectives. If $\mathcal{C}$ is also a dualizing $k$-variety, then $\mathcal{C}/[\mathcal{P}]$ and $\mathcal{C}/[\mathcal{I}]$ are also dualizing $k$-varieties (see Example \ref{ex2.1}(d)). We have two dualities $\Phi:\textup{mod-}\mathcal{C}/[\mathcal{P}]\rightarrow \textup{mod-}(\mathcal{C}/[\mathcal{I}])^{\textup{op}}$ and $D:\textup{mod-}(\mathcal{C}/[\mathcal{I}])^{\textup{op}}\rightarrow\textup{mod-}\mathcal{C}/[\mathcal{I}]$. The composition of $\Phi$ and $D$ defines an equivalence
$$\Theta: \ \textup{mod-}\mathcal{C}/[\mathcal{P}] \xrightarrow{\Phi}  \textup{mod-}(\mathcal{C}/[\mathcal{I}])^{\textup{op}} \xrightarrow{D}\textup{mod-}\mathcal{C}/[\mathcal{I}].$$
We consider the following restriction
$$\Theta: \ \textup{proj-}\mathcal{C}/[\mathcal{P}] \xrightarrow{\Phi}  \textup{inj-}(\mathcal{C}/[\mathcal{I}])^{\textup{op}} \xrightarrow{D}\textup{proj-}\mathcal{C}/[\mathcal{I}].$$
%$$\ \ \ \mathcal{C}/[\mathcal{P}](X,-) \mapsto   \textup{Ext}^1_\mathcal{C}(-,\tau X)\mapsto   \mathcal{C}/[\mathcal{I}](\tau X,-)$$
Since the projective object in mod-$\mathcal{C}/[\mathcal{I}]$ is of the form $\mathcal{C}/[\mathcal{I}](-,Y)$, we have
$$\Theta(\mathcal{C}/[\mathcal{P}](-,X))=D\textup{Ext}^1_\mathcal{C}(X,-)\cong\mathcal{C}/[\mathcal{I}](-,Y)$$
for some $Y\in\mathcal{C}$.
Therefore, there is an equivalence $\tau: \mathcal{C}/[\mathcal{P}]\cong\mathcal{C}/[\mathcal{I}]$ mapping $X$ to $Y$. The functor $\tau$ induces an equivalence $\tau^{-1}_\ast: \textup{mod-}\mathcal{C}/[\mathcal{P}]\cong \textup{mod-}\mathcal{C}/[\mathcal{I}], F\mapsto F\tau^{-1}$, such that $D\Phi=\tau^{-1}_\ast$. Assume that $\delta$ is a short exact sequence in $\mathcal{E}$, then $D\Phi(\delta^\ast)=D\delta_\ast$. On the other hand, $\tau^{-1}_\ast(\delta^\ast)=\delta^\ast\tau^{-1}$. Hence, we have $D\delta_\ast=\delta^\ast\tau^{-1}$.

%Similarly, the composition of $\Phi:\textup{mod-}\mathcal{C}/[\mathcal{P}]\rightarrow\textup{mod-}(\mathcal{C}/[\mathcal{I}])^{\textup{op}}$ and $D:\textup{mod-}(\mathcal{C}/[\mathcal{I}])^{\textup{op}}\rightarrow\textup{mod-}\mathcal{C}/[\mathcal{I}]$ defines an equivalence $\Theta':\textup{mod-}\mathcal{C}/[\mathcal{P}]\simeq\textup{mod-}\mathcal{C}/[\mathcal{I}]$. Thus,
%$$\Theta'(\mathcal{C}/[\mathcal{P}](-,X))=D\textup{Ext}^1_\mathcal{C}(X,-)=\mathcal{C}/[\mathcal{I}](-,Y)$$
%for some $Y$, which defines an equivalence $\tau:\mathcal{C}/[\mathcal{P}]\simeq\mathcal{C}/[\mathcal{I}], X\mapsto Y$.

\vspace{1mm}
To summarize,  we have the following generalized Auslander-Reiten duality and defect formula.

\begin{thm}\label{thm4.4.2}
Let $(\mathcal{C},\mathcal{E})$ be an Ext-finite exact category with enough projectives and injectives. Assume that $\mathcal{C}$ is a dualizing $k$-variety. Then there is an equivalence $\tau:\mathcal{C}/[\mathcal{P}]\cong\mathcal{C}/[\mathcal{I}]$ satisfying the following properties:

\textup{(a)} $D\textup{Ext}^1_\mathcal{C}(-,X)\cong\mathcal{C}/[\mathcal{P}](\tau^{-1} X,-)$, $D\textup{Ext}^1_\mathcal{C}(X,-)\cong\mathcal{C}/[\mathcal{I}](-,\tau X)$.

\textup{(b)} $D\delta_\ast=\delta^\ast\tau^{-1}$, $D\delta^\ast=\delta_\ast\tau$ for each short exact sequence $\delta$ in $\mathcal{E}$.

\noindent Therefore, $\mathcal{C}$ has Auslander-Reiten sequences.
\end{thm}

\begin{proof}
We have shown that there is an equivalence $\tau:\mathcal{C}/[\mathcal{P}]\cong\mathcal{C}/[\mathcal{I}]$ such that $D\textup{Ext}^1_\mathcal{C}(X,-)=\mathcal{C}/[\mathcal{I}](-,\tau X)$ and $D\delta_\ast=\delta^\ast\tau^{-1}$. Then $D\delta^\ast=\delta_\ast\tau$ immediately follows from  $D\delta_\ast=\delta^\ast\tau^{-1}$.
If $\delta=I_X$, then $\delta^\ast=\textup{Ext}^1_\mathcal{C}(-,X)$ and $\delta_\ast\tau=\mathcal{C}/[\mathcal{I}](X,\tau-)\cong\mathcal{C}/[\mathcal{P}](\tau^{-1}X,-)$. Since $D\delta^\ast=\delta_\ast\tau$, we have $D\textup{Ext}^1_\mathcal{C}(-,X)\cong\mathcal{C}/[\mathcal{P}](\tau^{-1}X,-)$.

The last assertion follows from (a) and \cite[Theorem 1.1]{[LZ]}; see \cite[Theorem 3.6]{[LNP]} for the exact version.
\end{proof}

\begin{cor}(\cite[Theorem 7.1.3]{[Rei]})
Let $\mathcal{C}$ be a dualizing $k$-variety, then mod-$\mathcal{C}$ has Auslander-Reiten sequences.
\end{cor}

\begin{proof}
By Example 2.5(b), mod-$\mathcal{C}$ is a dualizing $k$-variety, moreover, it is an abelian category with enough projectives and enough injectives. Thus the consequence follows from Theorem \ref{thm4.4.2}.
\end{proof}

\begin{rem}
Let $A$ be an Artin $k$-algebra and $\mathcal{C}=\textup{mod-}A$. Then $\mathcal{C}$ satisfies all the conditions in Theorem \ref{thm4.4.2} and the functor $\tau:\mathcal{C}/[\mathcal{P}]\cong\mathcal{C}/[\mathcal{I}]$  is given by $DTr$.
\end{rem}

\subsection{Higher version}

In this subsection, we assume that $n$ is an integer greater than or equal to 1 and $(\mathcal{C},\mathcal{E}_n)$ is an $n$-exact category in the sense of Jasso (see \cite{[J]}). Let  $X_\bullet: 0\rightarrow X_1\xrightarrow{f_1} X_2\xrightarrow{f_2}\cdots\xrightarrow{f_{n+1}} X_{n+2}\rightarrow 0$ be an $n$-exact sequence in $\mathcal{E}_n$, we say $f_1$ is an {\em admissible monomorphism} and $f_{n+1}$ is an {\em admissible epimorphism}. We always view $X_\bullet$ as a complex over $\mathcal{C}$ concentrated on degree $1,2,\cdots,n+2$. We denote by $\mathcal{E}_n(\mathcal{C})$ the category of all $n$-exact sequences in $\mathcal{E}_n$ where the morphisms between two $n$-exact sequences are given by morphisms of complexes, and by  $C\mathcal{E}_n(\mathcal{C})$ the full subcategory of $\mathcal{E}_n(\mathcal{C})$ formed by contractible $n$-exact sequences.

Assume that $(\mathcal{C},\mathcal{E}_n)$ is an $n$-exact category. An object $P\in\mathcal{C}$ is {\em projective} if for each admissible epimorphism $f:X\rightarrow Y$ the sequence $\mathcal{C}(P,X)\rightarrow\mathcal{C}(P,Y)\rightarrow 0$ is exact. The full subcategory of $\mathcal{C}$ formed by projectives is denoted by $\mathcal{P}$. We say  $(\mathcal{C},\mathcal{E}_n)$ {\em has enough projectives} if for each object $X\in\mathcal{C}$ there is an $n$-exact sequence $P_X: 0\rightarrow Y\rightarrow P_n\rightarrow\cdots\rightarrow P_1\rightarrow X\rightarrow 0$ in $\mathcal{E}_n$ with $P_i\in \mathcal{P}$.
We denote by $P\mathcal{E}_n(\mathcal{C})$ the full subcategory of $\mathcal{E}_n(\mathcal{C})$ formed by $P_X$ and contractible $n$-exact sequences $X_\bullet$ with $f_{n+1}=1$. Dually, we define {\em injective objects} and $(\mathcal{C},\mathcal{E}_n)$ {\em has enough injectives}. The full subcategory of $\mathcal{C}$ formed by injectives is denoted by $\mathcal{I}$ . We denote by $I\mathcal{E}_n(\mathcal{C})$ the full subcategory of $\mathcal{E}_n(\mathcal{C})$ formed by $I_X: 0\rightarrow X\rightarrow I_1\rightarrow\cdots\rightarrow I_n\rightarrow Y\rightarrow 0$ with $I_i\in \mathcal{I}$ and contractible $n$-exact sequences $X_\bullet$ with $f_{1}=1$.

With notations as above, we have the following higher version of Theorem \ref{thm4.1}.

\begin{thm} \label{thm4.6}
Let $(\mathcal{C},\mathcal{E}_n)$ be an $n$-exact category.

\textup{(a)} If $(\mathcal{C},\mathcal{E}_n)$ has enough projectives, %Denote by $\mathcal{P}$ the full subcategory of $\mathcal{C}$ formed by projectives.
then we have the following equivalences:
$$\mathcal{E}_n(\mathcal{C})/[C\mathcal{E}_n(\mathcal{C})]\cong \textup{mod-}\mathcal{C}/[\mathcal{P}],
\mathcal{E}_n(\mathcal{C})/[P\mathcal{E}_n(\mathcal{C})]\cong (\textup{mod-}(\mathcal{C}/[\mathcal{P}])^{\textup{op}})^{\textup{op}}.$$
 %where $\mathcal{P}$ is the full subcategory of $\mathcal{C}$ formed by projective objects.

\textup{(b)} If $(\mathcal{C},\mathcal{E}_n)$ has enough injectives, %Denote by $\mathcal{P}$ the full subcategory of $\mathcal{C}$ formed by projectives.
 then we have the following equivalences:
$$\mathcal{E}_n(\mathcal{C})/[C\mathcal{E}_n(\mathcal{C})]\cong (\textup{mod-}(\mathcal{C}/[\mathcal{I}])^{\textup{op}})^{\textup{op}},
\mathcal{E}_n(\mathcal{C})/[I\mathcal{E}_n(\mathcal{C})]\cong\textup{mod-}\mathcal{C}/[\mathcal{I}].$$
\end{thm}

\begin{lem} (\cite[Proposition 4.8]{[J]})\label{lem4.5.1}
Let $(\mathcal{C},\mathcal{E}_n)$ be an $n$-exact category. If the sequence
$$0\rightarrow X_1\xrightarrow{f_1}X_2\xrightarrow{f_2}\cdots\xrightarrow{f_n}X_{n+1}\xrightarrow{f_{n+1}} X_{n+2}\rightarrow 0$$
 is an $n$-exact sequence in $\mathcal{E}_n$, then the following statements are equivalent.

\textup{(a)} The diagram
$$\xymatrix{
 X_1 \ar[r]^{f_1}\ar[d]^{\varphi_1} & X_2 \ar[r]^{f_2}\ar[d]^{\varphi_2} & \cdots \ar[r]^{f_{n-1}}& X_{n}\ar[r]^{f_{n}}\ar[d]^{\varphi_{n}}& X_{n+1} \ar[d]^{\varphi_{n+1}} \\
 Y_1 \ar[r]^{g_1} & Y_2 \ar[r]^{g_2} & \cdots\ar[r]^{g_{n-1}}& Y_{n} \ar[r]^{g_{n}} & Y_{n+1} \\
}$$  is an $n$-pushout and $n$-pullback diagram.

\textup{(b)} The sequence $$\xymatrixcolsep{3.5pc}\xymatrix{
0\ar[r] & X_1\ar[r]^{\left(
             \begin{smallmatrix}
               -f_1 \\
               \varphi_1 \\
             \end{smallmatrix}
           \right)
}& X_2\oplus Y_1\ar[r]^{\left(
             \begin{smallmatrix}
              -f_2 & 0 \\
               \varphi_2 & g_1 \\
             \end{smallmatrix}
           \right)}& X_3\oplus Y_2\ar[r]^{\left(
             \begin{smallmatrix}
              -f_3 & 0 \\
               \varphi_3 & g_2 \\
             \end{smallmatrix}
           \right)}& \cdots\\
}$$ $$\begin{gathered}\xymatrixcolsep{3.5pc}\xymatrix{
\cdots\ar[r]^{\left(
             \begin{smallmatrix}
              -f_{n} & 0 \\
               \varphi_{n} & g_{n-1} \\
             \end{smallmatrix}
           \right)\ \ \ \ }& X_{n+1}\oplus Y_{n}\ar[r]^{\ \ \ \left(
             \begin{smallmatrix}
               \varphi_{n+1} & g_{n} \\
             \end{smallmatrix}
           \right)}& Y_{n+1} \ar[r] & 0
}\end{gathered}$$ is an $n$-exact sequence in $\mathcal{E}_n$.

\textup{(c)} There exists a commutative diagram
$$\xymatrix{
0\ar[r] & X_1 \ar[r]^{f_1}\ar[d]^{\varphi_1} & X_2 \ar[r]^{f_2}\ar[d]^{\varphi_2} & \cdots \ar[r]^{f_{n-1}}& X_{n}\ar[r]^{f_{n}}\ar[d]^{\varphi_{n}}& X_{n+1} \ar[d]^{\varphi_{n+1}}\ar[r]^{f_{n+1}} & X_{n+2}\ar@{=}[d]\ar[r] & 0\\
0\ar[r] & Y_1 \ar[r]^{g_1} & Y_2 \ar[r]^{g_2} & \cdots\ar[r]^{g_{n-1}}& Y_{n} \ar[r]^{g_{n}} & Y_{n+1}\ar[r]^{g_{n+1}} & X_{n+2} \ar[r] & 0 \\
}$$
such that the second row is an $n$-exact sequence in $\mathcal{E}_n$.
\end{lem}

The following lemma is a higher version of Lemma \ref{lem0}.

\begin{lem}\label{lem 3.1}
Assume that the following diagram
$$\xymatrix{X_\bullet\ar[d]^{\varphi_\bullet} & 0\ar[r] & X_1 \ar[r]^{f_1}\ar[d]^{\varphi_1} & X_2 \ar[r]^{f_2}\ar[d]^{\varphi_2}  &\cdots \ar[r]^{f_n}  & X_{n+1} \ar[r]^{f_{n+1}}\ar[d]^{\varphi_{n+1}} & X_{n+2} \ar[r]\ar[d]^{\varphi_{n+2}} & 0\\
Y_\bullet & 0\ar[r] & Y_1 \ar[r]^{g_1} & Y_2 \ar[r]^{g_2} & \cdots \ar[r]^{g_n} & Y_{n+1} \ar[r]^{g_{n+1}}& Y_{n+2} \ar[r] & 0 \\
}$$
is commutative with rows in $\mathcal{E}_n$. Then we get the following commutative diagram
$$\xymatrixcolsep{1.6pc}
\xymatrix{
K(\varphi_\bullet)\ar[d]^{k_\bullet} & 0\ar[r] & X_1 \ar[r]^{\left(
                                                                                        \begin{smallmatrix}
                                                                                          -f_1  \\
                                                                                           \varphi_1\\
                                                                                        \end{smallmatrix}
                                                                                      \right)}\ar@{=}[d] & X_2\oplus Y_1 \ar[r]^{\left(
                                                                                        \begin{smallmatrix}
                                                                                        -f_2 & 0 \\
                                                                                          \pi_2 & h_1 \\
                                                                                        \end{smallmatrix}
                                                                                      \right)}\ar[d]^{(-1,0)}
& \cdots \ar[r]^{\left(
                                                                                        \begin{smallmatrix}
                                                                                        -f_n & 0 \\
                                                                                          \pi_n & h_{n-1} \\
                                                                                        \end{smallmatrix}
                                                                                      \right)} & X_{n+1}\oplus Z_n \ar[r]^{(\pi_{n+1,h_n})}\ar[d]^{((-1)^n,0)} & Z_{n+1}\ar[r]\ar[d]^{(-1)^nh_{n+1}} & 0\\
X_\bullet\ar[d]^{\pi_\bullet} & 0\ar[r] & X_1 \ar[r]^{f_1}\ar[d]^{\varphi_1} & X_2 \ar[r]^{f_2}\ar[d]^{\pi_2} &\cdots \ar[r]^{f_n} & X_{n+1} \ar[r]^{f_{n+1}}\ar[d]^{\pi_{n+1}} & X_{n+2} \ar[r]\ar@{=}[d] & 0\\
I(\varphi_\bullet)\ar[d]^{i_\bullet} & 0\ar[r] & Y_1 \ar[r]^{h_1}\ar@{=}[d] & Z_2 \ar[r]^{h_2}\ar[d]^{i_2} &\cdots \ar[r]^{h_n} & Z_{n+1}\ar[r]^{h_{n+1}}\ar[d]^{i_{n+1}} & X_{n+2} \ar[r]\ar[d]^{\varphi_{n+2}} & 0\\
Y_\bullet\ar[d]^{c_\bullet} & 0\ar[r] & Y_1 \ar[r]^{g_1}\ar[d]^{h_1} & Y_2 \ar[r]^{g_2}\ar[d]^{\left(
                                                                                                                            \begin{smallmatrix}
                                                                                                                              0 \\
                                                                                                                              1 \\
                                                                                                                            \end{smallmatrix}
                                                                                                                          \right)}
& \cdots \ar[r]^{g_n} & Y_{n+1}\ar[r]^{g_{n+1}}\ar[d]^{\left(
                                                         \begin{smallmatrix}
                                                           0 \\
                                                           1 \\
                                                         \end{smallmatrix}
                                                       \right)
} & Y_{n+2} \ar[r]\ar@{=}[d] & 0\\
C(\varphi_\bullet) & 0 \ar[r] &  Z_2 \ar[r]^{\left(
                                                                                        \begin{smallmatrix}
                                                                                          -h_2  \\
                                                                                           i_2\\
                                                                                        \end{smallmatrix}
                                                                                      \right)} & Z_3\oplus Y_2 \ar[r]^{\left(
                                                                                        \begin{smallmatrix}
                                                                                          -h_3 & 0 \\
                                                                                          i_3 & g_2\\
                                                                                        \end{smallmatrix}
                                                                                      \right)} & \cdots \ar[r]^{\left(
                                                                                        \begin{smallmatrix}
                                                                                          -h_{n+1} & 0 \\
                                                                                          i_{n+1} & g_{n}\\
                                                                                        \end{smallmatrix}
                                                                                      \right)}  & X_{n+2}\oplus Y_{n+1} \ar[r]^{(\varphi_{n+2},g_{n+1})} & Y_{n+2} \ar[r] & 0\\
}$$
 with rows in $\mathcal{E}_n$, moreover,  $\underline{\varphi_\bullet}=\underline{i_\bullet}\underline{\pi_\bullet}$ in $\mathcal{E}_n(\mathcal{C})/[C\mathcal{E}_n(\mathcal{C})]$.
\end{lem}

\begin{proof}
By \cite[Proposition 4.9]{[J]}, there exist two morphisms $\pi_\bullet:X_\bullet\rightarrow I(\varphi_\bullet)$ and $i_\bullet:I(\varphi_\bullet)\rightarrow Y_\bullet$ such that $I(\varphi_\bullet)\in \mathcal{E}_n$ and $\underline{\varphi_\bullet}=\underline{i_\bullet}\underline{\pi_\bullet}$ in $\mathcal{E}_n(\mathcal{C})/[C\mathcal{E}_n(\mathcal{C})]$. The sequences $K(\varphi_\bullet)$ and $C(\varphi_\bullet)$ are $n$-exact sequences in $\mathcal{E}_n$ by Lemma \ref{lem4.5.1}.
\end{proof}

\begin{thm}
Let $(\mathcal{C},\mathcal{E}_n)$ be an $n$-exact category. Then the  category $\mathcal{E}_n(\mathcal{C})/[C\mathcal{E}_n(\mathcal{C})]$
is an abelian category whose kernels and cokernels are given by $n$-pullback and $n$-pushout diagrams.
\end{thm}

\begin{proof}
The proof is similar to that of Theorem \ref{thm4.3.1}.
\end{proof}

\begin{defn}(see \cite{[Iy]})
An $n$-exact sequence $0\rightarrow X_1\xrightarrow{f_1} X_2\xrightarrow{f_2}\cdots\xrightarrow{f_{n+1}} X_{n+2}\rightarrow 0$ is called an {\em $n$-Auslander-Reiten sequence} if the following conditions are satisfied:

(a) All $f_i\in J_\mathcal{C}$.

(b) If $g:X_1\rightarrow Y$ is not a section, then $g$ factors through $f_1$.

(c) If $h:Z\rightarrow X_{n+2}$ is not a retraction, then $h$ factors through $f_{n+1}$.

\end{defn}

The following is a higher analogue of Theorem \ref{thm4.4.1}.

\begin{thm}
Let $(\mathcal{C},\mathcal{E}_n)$ be an Ext-finite $k$-linear $n$-exact category.

\textup{(a)} Assume that $X_\bullet: 0\rightarrow X_1\xrightarrow{f_1} X_2\xrightarrow{f_2}\cdots\xrightarrow{f_{n+1}} X_{n+2}\rightarrow 0$ is an $n$-exact sequence in $\mathcal{E}_n$ where $f_i\in J_\mathcal{C}$, $X_1$ and $X_{n+2}$ are indecomposable. Then $X_\bullet$ is simple in $\mathcal{E}_n(\mathcal{C})/[C\mathcal{E}_n(\mathcal{C})]$ if and only if $X_\bullet$ is an $n$-Auslander-Reiten sequence in $\mathcal{C}$.

\textup{(b)} There is a bijection between the set of isoclasses of simple objects in $\mathcal{E}_n(\mathcal{C})/[C\mathcal{E}_n(\mathcal{C})]$ and the set of isoclasses of $n$-Auslander-Reiten sequences in $\mathcal{C}$.
\end{thm}

\begin{prop} Let $(\mathcal{C},\mathcal{E}_n)$ be an $n$-exact category.

\textup{(a)} Each $n$-exact sequence $P_{X}: 0\rightarrow Y\xrightarrow{f_1}P_n\xrightarrow{f_2}\cdots\xrightarrow{f_{n-1}} P_2\xrightarrow{f_n} P_1\xrightarrow{f_{n+1}} X\rightarrow 0$ in $\mathcal{E}_n$ with $P_i$  projective  is a projective object in $\mathcal{E}_n(\mathcal{C})/[C\mathcal{E}_n(\mathcal{C})]$.

\textup{(b)} If $(\mathcal{C},\mathcal{E}_n)$ has enough projectives, then each projective object in $\mathcal{E}_n(\mathcal{C})/[C\mathcal{E}_n(\mathcal{C})]$ is of the form $P_X$ for some object $X$ in $\mathcal{C}$. In this case, $\mathcal{E}_n(\mathcal{C})/[C\mathcal{E}_n(\mathcal{C})]$ has enough projectives.
%Dually, each $n$-exact sequence $I_{X}: 0\rightarrow X\xrightarrow{f_1} I\xrightarrow{f_2} X_3\xrightarrow{f_3}\cdots\xrightarrow{f_{n+1}} X_{n+2}\rightarrow 0$ in $\mathcal{E}$ with $I$  injective is an injective object in $\mathcal{E}_n(\mathcal{C})/[C\mathcal{E}_n(\mathcal{C})]$.
\end{prop}

The following is a combination of \cite[Theorem 1.3]{[Kva]} and \cite[Lemma 3.5]{[Iy0]}.

\begin{lem} \label{lem4.6.1}
Let $\mathcal{C}$ be an $n$-abelian category with enough projectives.

\textup{(a)} There exists an abelian category $\mathcal{A}$ with enough projectives such that $\mathcal{C}$ is an $n$-cluster-tilting subcategory of $\mathcal{A}$, moreover, the class of  projectives in $\mathcal{C}$ and the class of projectives in $\mathcal{A}$ coincide.

\textup{(b)} For each $n$-exact sequence $$ 0\rightarrow X_1\xrightarrow{f_1} X_2\xrightarrow{f_2}\cdots\xrightarrow{f_n} X_{n+1}\xrightarrow{f_{n+1}}X_{n+2}\rightarrow 0,$$
there exist two long exact sequences
$$\textup{(i)}\ 0\rightarrow\mathcal{C}(-,X_{1})\rightarrow\mathcal{C}(-,X_{2})\rightarrow\cdots
\rightarrow\mathcal{C}(-,X_{n+2})\rightarrow$$
$$\rightarrow\textup{Ext}^n_\mathcal{A}(-,X_1)|_{\mathcal{C}}\rightarrow\textup{Ext}^n_\mathcal{A}(-,X_2)|_{\mathcal{C}}\rightarrow
\cdots\rightarrow\textup{Ext}^n_\mathcal{A}(-,X_{n+2})|_{\mathcal{C}}.$$
%\rightarrow\textup{Ext}^{2n}_\mathcal{A}(-,X_1)|_{\mathcal{C}}\rightarrow\cdots.$$
$$\textup{(ii)}\ 0\rightarrow\mathcal{C}(X_{n+2},-)\rightarrow\mathcal{C}(X_{n+1},-)\rightarrow\cdots
\rightarrow\mathcal{C}(X_{1},-)\rightarrow$$
$$\rightarrow\textup{Ext}^n_\mathcal{A}(X_{n+2},-)|_{\mathcal{C}}\rightarrow\textup{Ext}^n_\mathcal{A}(X_{n+1},-)|_{\mathcal{C}}\rightarrow
\cdots\rightarrow\textup{Ext}^n_\mathcal{A}(X_{1},-)|_{\mathcal{C}}.$$ %\rightarrow\textup{Ext}^{2n}_\mathcal{A}(-,X_1)|_{\mathcal{C}}\rightarrow\cdots.$$
\end{lem}

%\begin{rem}
From now on to the end of this section, we assume that $\mathcal{C}$ is an $n$-abelian category with enough projectives and  injectives.  The full subcategory of $\mathcal{C}$ consisting of projectives (resp. injectvies) is denoted by $\mathcal{P}$ (resp. $\mathcal{I}$).  By Lemma \ref{lem4.6.1}, we always view $\mathcal{C}$ as an $n$-cluster tilting subcategory of an abelian category $\mathcal{A}$. For convenience, we denote $\textup{Ext}^n_\mathcal{A}(-,X)|_\mathcal{C}$ by $\textup{Ext}^n_\mathcal{C}(-,X)$ and denote $\textup{Ext}^n_\mathcal{A}(X,-)|_\mathcal{C}$ by $\textup{Ext}^n_\mathcal{C}(X,-)$ for short.
%\end{rem}

\begin{defn} (\cite[Definition 3.1]{[JK]}) Given an $n$-exact sequence
$$\delta:  0\rightarrow X_1\xrightarrow{f_1} X_2\xrightarrow{f_2}\cdots\xrightarrow{f_n} X_{n+1}\xrightarrow{f_{n+1}}X_{n+2}\rightarrow 0,$$
we define  the {\em contravariant defect} $\delta^\ast$  and the {\em covariant defect} $\delta_\ast$ by the following exact sequence of functors
$$0\rightarrow\mathcal{C}(-,X_{1})\xrightarrow{\mathcal{C}(-,f_{1})}\mathcal{C}(-,X_{2})\xrightarrow{\mathcal{C}(-,f_{2})}\cdots
\xrightarrow{\mathcal{C}(-,f_{n+1})}\mathcal{C}(-,X_{n+2})\rightarrow \delta^\ast\rightarrow 0,$$
$$0\rightarrow\mathcal{C}(X_{n+2},-)\xrightarrow{\mathcal{C}(f_{n+1},-)}\cdots\xrightarrow{\mathcal{C}(f_{2},-)}
\mathcal{C}(X_{2},-)\xrightarrow{\mathcal{C}(f_{1},-)}\mathcal{C}(X_1,-)\rightarrow \delta_\ast\rightarrow 0.$$
\end{defn}

\begin{example}
(a) Let $\delta=P_{X}: 0\rightarrow Y\xrightarrow{f_1}P_n\xrightarrow{f_2}\cdots\xrightarrow{f_n} P_1\xrightarrow{f_{n+1}} X\rightarrow 0$ with $P_i\in\mathcal{P}$, then $\delta^*=\mathcal{C}/[\mathcal{P}](-,X)$ and $\delta_*=\textup{Ext}^n_\mathcal{C}(X,-)$.

(b) Let $\delta=I_{X}: 0\rightarrow X\xrightarrow{f_1}I_1\xrightarrow{f_2}\cdots\xrightarrow{f_n} I_n\xrightarrow{f_{n+1}} Y\rightarrow 0$ with $I_i\in\mathcal{I}$, then $\delta^*=\textup{Ext}^n_\mathcal{C}(-,X)$ and $\delta_*=\mathcal{C}/[\mathcal{I}](X,-)$.
\end{example}

\begin{rem}
 (a) In Theorem \ref{thm4.6}, the equivalence  $\alpha:\mathcal{E}_n(\mathcal{C})/[C\mathcal{E}_n(\mathcal{C})]\cong \textup{mod-}\mathcal{C}/[\mathcal{P}]$ is given by $\delta\mapsto \delta^\ast$, and the equivalence $\beta:\mathcal{E}_n(\mathcal{C})/[C\mathcal{E}_n(\mathcal{C})]\cong (\textup{mod-}(\mathcal{C}/[\mathcal{I}])^{\textup{op}})^{\textup{op}}$ is given by $\delta\mapsto \delta_\ast$.

(b) In $\textup{mod-}\mathcal{C}/[\mathcal{P}]$, each projective object is of the form $\mathcal{C}/[\mathcal{P}](-,X)$, and each injective object is of the form $\textup{Ext}^n_\mathcal{C}(-,X)$.
\end{rem}

%\begin{proof}
%(b) It is easy to check that $\alpha(P_X)=\mathcal{C}/[\mathcal{P}](-,X)$ and $\alpha(I_X)=\textup{Ext}^n_\mathcal{C}(-,X)$.
%\end{proof}

\begin{prop}
Let $\mathcal{C}$ be an $n$-abelian category with enough projectives and injectives. Then there is a duality
$$\Phi:\textup{mod-}\mathcal{C}/[\mathcal{P}]\rightarrow\textup{mod-}(\mathcal{C}/[\mathcal{I}])^{\textup{op}},\ \ \ \delta^\ast\mapsto \delta_\ast.$$
Moreover, by restrictions, we obtain the following two dualities
$$\Phi:\textup{proj-}\mathcal{C}/[\mathcal{P}]\rightarrow\textup{inj-}(\mathcal{C}/[\mathcal{I}])^{\textup{op}}, \ \ \ \mathcal{C}/[\mathcal{P}](-,X)\mapsto \textup{Ext}^n_\mathcal{C}(X,-).$$
$$\Phi:\textup{inj-}\mathcal{C}/[\mathcal{P}]\rightarrow\textup{proj-}(\mathcal{C}/[\mathcal{I}])^{\textup{op}}, \ \ \ \textup{Ext}^n_\mathcal{C}(-,X)\mapsto\mathcal{C}/[\mathcal{I}](X,-).$$
\end{prop}

\begin{thm} (Higher Hilton-Rees Theorem)
Let $\mathcal{C}$ be an $n$-abelian category with enough projectives and injectives.

\textup{(a)} There is an isomorphism between $\mathcal{C}/[\mathcal{P}](Y,X)$ and the group of natural transformations from $\textup{Ext}^n_\mathcal{C}(X,-)$ to $\textup{Ext}^n_\mathcal{C}(Y,-)$.

\textup{(b)} There is an isomorphism between $\mathcal{C}/[\mathcal{I}](X,Y)$ and the group of natural transformations from $\textup{Ext}^n_\mathcal{C}(-,X)$ to $\textup{Ext}^n_\mathcal{C}(-,Y)$.
\end{thm}

\begin{thm}\label{thm4.5.3}
Let $\mathcal{C}$ be an $n$-abelian category with enough projectives and injectives. Assume that $\mathcal{C}$ is a dualizing $k$-variety. Then there is an equivalence $\tau_n:\mathcal{C}/[\mathcal{P}]\cong\mathcal{C}/[\mathcal{I}]$ satisfying the following properties:

\textup{(a)} $D\textup{Ext}^n_\mathcal{C}(-,X)=\mathcal{C}/[\mathcal{P}](\tau_n^{-1} X,-)$, $D\textup{Ext}^n_\mathcal{C}(X,-)=\mathcal{C}/[\mathcal{I}](-,\tau_n X)$.

\textup{(b)} $D\delta_\ast=\delta^\ast\tau_n^{-1}$, $D\delta^\ast=\delta_\ast\tau_n$ for each $n$-exact sequence $\delta$.
\end{thm}

\begin{rem}
Let $A$ be an Artin $k$-algebra and $\mathcal{C}$ be an $n$-cluster-tilting subcategory of mod-$A$. Then $\mathcal{C}$ contains proj-$A$ and inj-$A$, moreover, $\mathcal{C}$ is a dualizing $k$-variety since it is functorially finite. Indeed,  the functor $\tau_n:\mathcal{C}/[\mathcal{P}]\cong\mathcal{C}/[\mathcal{I}]$ in Theorem \ref{thm4.5.3} is given by $DTr\Omega^{n-1}$.
\end{rem}

\section{Abeian quotients of the categories of triangles}

Let $\mathcal{C}$ be a triangulated category with the suspension functor $\Sigma$. We denote by $\Delta(\mathcal{C})$ the category of triangles in $\mathcal{C}$, where the objects are the triangles $X_\bullet=(X_1\xrightarrow{f_1} X_2\xrightarrow{f_2} X_3\xrightarrow{f_3}\Sigma X_1)$ and the morphisms from $X_\bullet$ to $Y_\bullet$ are the triples $\varphi_\bullet=(\varphi_1,\varphi_2,\varphi_3)$ such that the following diagram is commutative:
$$\xymatrix{
X_1 \ar[r]^{f_1} \ar[d]^{\varphi_1} & X_2 \ar[r]^{f_2}\ar[d]^{\varphi_2} & X_3 \ar[r]^{f_3}\ar[d]^{\varphi_3} & \Sigma X_1 \ar[d]^{\Sigma\varphi_1}\\
Y_1\ar[r]^{g_1} & Y_2\ar[r]^{g_2} & Y_3\ar[r]^{g_3} & \Sigma Y_1
}$$
 Let $X_\bullet$ and $Y_\bullet$ be two triangles, we denote by $\mathcal{R}_2(X_\bullet,Y_\bullet)$ (resp. $\mathcal{R}'_2(X_\bullet,Y_\bullet)$) the class of morphisms $\varphi_\bullet: X_\bullet\rightarrow Y_\bullet$ such that there is a morphism $p:X_3\rightarrow Y_2$ such that $g_2p=\varphi_3$ (resp. $pf_2=\varphi_2$). It is easy to see that $\mathcal{R}_2$  and $\mathcal{R}_2'$  are ideals of $\Delta(\mathcal{C})$.

The first part of the following result is implied in \cite{[Nee]}.

 \begin{thm}\label{thm5.1}
Let $\mathcal{C}$ be a triangulated category, then we have the following two equivalences.

\textup{(a)} $\Delta(\mathcal{C})/\mathcal{R}_2\cong \textup{mod-}\mathcal{C}$.

\textup{(b)} $\Delta(\mathcal{C})/\mathcal{R}'_2\cong (\textup{mod-}\mathcal{C}^{\textup{op}})^{\textup{op}}$.
 \end{thm}

 \begin{proof}
 Define a functor $\alpha:\Delta(\mathcal{C})\rightarrow\textup{Mor}(\mathcal{C})$ by taking a triangle $X_1\xrightarrow{f_1} X_2\xrightarrow{f_2} X_3\xrightarrow{f_3}\Sigma X_1$ to $f_2:X_2\rightarrow X_3$. It is routine to check that $\Delta(\mathcal{C})/\mathcal{R}_2\cong\textup{Mor}(\mathcal{C})/\mathcal{R}$ and $\Delta(\mathcal{C})/\mathcal{R}'_2\cong\textup{Mor}(\mathcal{C})/\mathcal{R'}$. Then the result follows from Lemma \ref{prop3.1}.
 \end{proof}

Let $(\mathcal{C},\mathcal{E})$ be a Frobenius category. We denote by $\mathcal{P}$ the full subcategory of $\mathcal{C}$ formed by projectives. It is well known that the quotient category $\mathcal{C}/[\mathcal{P}]$ is a triangulated category. The following corollary follows from Theorem \ref{thm4.1} and Theorem \ref{thm5.1}.

\begin{cor}
Let $(\mathcal{C},\mathcal{E})$ be a Frobenius category, then the categories
$\mathcal{E}(\mathcal{C})/[S\mathcal{E}(\mathcal{C})]$, $\Delta(\mathcal{C}/[\mathcal{P}])/\mathcal{R}_2$ and $\textup{mod-}\mathcal{C}/[\mathcal{P}]$ are equivalent.
\end{cor}

\begin{rem}
Let $\mathcal{C}$ be a triangulated category.  Assume that $X_\bullet$ and $Y_\bullet$ are two triangles. We denote by $\mathcal{R}_1(X_\bullet,Y_\bullet)$ (resp. $\mathcal{R}_3(X_\bullet,Y_\bullet)$) the class of morphisms $\varphi_\bullet:X_\bullet\rightarrow Y_\bullet$ such that there is a morphism $p:X_2\rightarrow Y_1$ (resp. $p:\Sigma X_1\rightarrow Y_3$) such that $g_1p=\varphi_2$ (resp. $g_3p=\Sigma\varphi_1$). Then $\mathcal{R}_1$ and $\mathcal{R}_3$ are ideals of $\Delta(\mathcal{C})$. Moreover, we have equivalences $\Delta(\mathcal{C})/\mathcal{R}_1\cong \Delta(\mathcal{C})/\mathcal{R}_2\cong\Delta(\mathcal{C})/\mathcal{R}_3$, which are given by rotations. We can see Remark \ref{rem4.1} for comparison.

 Similarly, we have equivalences $\Delta(\mathcal{C})/\mathcal{R}'_1\cong \Delta(\mathcal{C})/\mathcal{R}'_2\cong\Delta(\mathcal{C})/\mathcal{R}'_3$, where  $\mathcal{R}'_1(X_\bullet,Y_\bullet)$ (resp. $\mathcal{R}'_3(X_\bullet,Y_\bullet)$) is the class of morphisms $\varphi_\bullet:X_\bullet\rightarrow Y_\bullet$ such that there is a morphism $p:X_2\rightarrow Y_1$ (resp. $p:\Sigma X_1\rightarrow Y_3$) such that $pf_1=\varphi_1$ (resp. $pf_3=\varphi_3$).
\end{rem}

From now on, we assume that $\mathcal{C}$ is a triangulated category. We will give some basic properties on the abelian category $\Delta(\mathcal{C})/\mathcal{R}_2$.

\begin{prop}\label{prop5.1}
Let $\varphi_\bullet:X_\bullet\rightarrow Y_\bullet$ be a morphism in $\Delta(\mathcal{C})$. Then we have

\textup{(a)} The following statements are equivalence:

\ \ \textup{(i)} $\underline{\varphi_\bullet}=0$ in $\Delta(\mathcal{C})/\mathcal{R}_2$.

\ \ \textup{(ii)} $\varphi_3$ factors through $g_2$.

\ \ \textup{(iii)} $g_3\varphi_3=0$.

\ \ \textup{(iv)} $(\Sigma\varphi_1)f_3=0$.

\ \ \textup{(v)} $\varphi_1$ factors through $f_1$.

\ \ \textup{(vi)} $\underline{\varphi_\bullet}=0$ in $\Delta(\mathcal{C})/\mathcal{R}'_1$.

%(b) The projective objects in $\Delta(\mathcal{C})/\mathcal{R}_2$ are of the form $X\rightarrow 0\rightarrow \Sigma X\xrightarrow{1} \Sigma X$.
%So are the injective objects.

\textup{(b)} The zero objects in $\Delta(\mathcal{C})/\mathcal{R}_2$ are of the form
$(X\xrightarrow{1} X\rightarrow 0\rightarrow \Sigma X)\oplus (0\rightarrow Y\xrightarrow{1} Y\rightarrow 0).$

\textup{(c)} $\underline{\varphi_\bullet}$ is a monomorphism in $\Delta(\mathcal{C})/\mathcal{R}_2$ if and only if $\left(
                                                                                     \begin{smallmatrix}
                                                                                       f_1 \\
                                                                                       \varphi_1 \\
                                                                                     \end{smallmatrix}
                                                                                   \right)
$ is a section.
\end{prop}

\begin{proof}
(a) It is clear.

(b) Assume that $X_\bullet: X_1\xrightarrow{f_1} X_2\xrightarrow{f_2} X_3\xrightarrow{f_3}\Sigma X_1$ is a zero object in $\Delta(\mathcal{C})/\mathcal{R}_2$. Then $\underline{\mbox{Id}_{X_\bullet}}=0$, thus $f_3=0$ by (a). Therefore, $X_\bullet$ is isomorphic to $(X_1\xrightarrow{1} X_1\rightarrow 0\rightarrow \Sigma X_1)\oplus (0\rightarrow X_3\xrightarrow{1} X_3\rightarrow 0).$

(c) The proof is similar to that of Lemma \ref{lem4.2}.
\end{proof}

\begin{rem}
(a) Denote by $\mathcal{U}$ the full subcategory of $\Delta(\mathcal{C})$ formed by $(X\xrightarrow{1} X\rightarrow 0\rightarrow \Sigma X)\oplus (0\rightarrow Y\xrightarrow{1} Y\rightarrow 0)$. By Proposition \ref{prop5.1}(b) there is a dense functor $\beta:\Delta(\mathcal{C})/[\mathcal{U}]\rightarrow \Delta(\mathcal{C})/\mathcal{R}_2$. But we point out that $\beta$ is not an equivalence in general, because a morphism $\varphi_\bullet: X_\bullet\rightarrow Y_\bullet$ in $\Delta(\mathcal{C})$ such that $\underline{\varphi_\bullet}=0$ in $\Delta(\mathcal{C})/\mathcal{R}_2$ does not imply that $\varphi_\bullet$ factors through some object in $\mathcal{U}$.

%We can compare it with Lemma \ref{rem3.1}.
(b) Assume that there is a commutative diagram
$$\xymatrix{
X_1 \ar[r]^{f_1} \ar[d]^{\varphi_1} & X_2 \ar[r]^{f_2}\ar[d]^{\varphi_2} & X_3 \ar[r]^{f_3} & \Sigma X_1 \ar[d]^{\Sigma\varphi_1}\\
Y_1\ar[r]^{g_1} & Y_2\ar[r]^{g_2} & Y_3\ar[r]^{g_3} & \Sigma Y_1
}$$
whose rows are triangles. It is well known that there is a morphism $\varphi_3:X_3\rightarrow Y_3$ such that the above diagram is commutative. But the morphism $\varphi_3$ is not unique in general. Assume that $\varphi'_3:X_3\rightarrow Y_3$ is another morphism satisfying required condition. Set $\varphi_\bullet=(\varphi_1,\varphi_2,\varphi_3)$ and $\varphi'_\bullet=(\varphi_1,\varphi_2,\varphi'_3)$. Then we have $\underline{\varphi_\bullet}=\underline{\varphi_\bullet'}$ in $\Delta(\mathcal{C})/\mathcal{R}_2$ by Proposition \ref{prop5.1}(a).
\end{rem}

Recall that a commutative diagram $$\xymatrix{
X_1 \ar[r]^{f_1} \ar[d]^{\varphi_1} & X_2 \ar[d]^{\varphi_2}\\
Y_1 \ar[r]^{g_1} & Y_2}$$ is called a {\em homotopy cartesian} if
$$X_1\xrightarrow{\left(
                    \begin{smallmatrix}
                      f_1 \\
                      \varphi_1 \\
                    \end{smallmatrix}
                  \right)
} X_2\oplus Y_1\xrightarrow{(\varphi_2,-g_1)} Y_2\xrightarrow{\delta}\Sigma X_1$$ is a triangle, where $\delta$ is called a {\em differential}.

The following result is well known, for example, see \cite[Appendix A]{[Kr]}.

\begin{lem}\label{lem5.1}
Let $(\mathcal{C},\Delta,\Sigma)$ be a pre-triangulated category. Then $\Delta$ satisfies axiom (TR4) if and only if for each commutative diagram
$$ \xymatrix{
 X_1 \ar[r]^{f_1}\ar@{=}[d] & X_2 \ar[r]^{f_2}\ar[d]^{\varphi_2} & X_3 \ar[r]^{f_3} & \Sigma X_1\ar@{=}[d]\\
 X_1 \ar[r]^{g_1} & Y_2 \ar[r]^{g_2} & Y_3 \ar[r]^{g_3} & \Sigma X_1\\
}$$ with rows in $\Delta$, there exists a morphism $\varphi_3:X_3\rightarrow Y_3$ such that the whole diagram is commutative and
the following diagram $$\xymatrix{
X_2\ar[r]^{f_2} \ar[d]^{\varphi_2} & X_3 \ar[d]^{\varphi_3}\\
Y_2 \ar[r]^{g_2} & Y_3}$$ is a homotopy cartesian.
\end{lem}

\begin{lem}\label{lem5.2}
Assume that the following
$$\xymatrix{X_\bullet\ar[d]^{\varphi_\bullet}  & X_1 \ar[r]^{f_1}\ar[d]^{\varphi_1} & X_2 \ar[r]^{f_2}\ar[d]^{\varphi_2} & X_3 \ar[r]^{f_3}\ar[d]^{\varphi_3} & \Sigma X_1\ar[d]^{\Sigma\varphi_1}\\
Y_\bullet  & Y_1 \ar[r]^{g_1} & Y_2 \ar[r]^{g_2} & Y_3 \ar[r]^{g_3} & \Sigma Y_1\\
}$$
is a morphism of triangles. Then we have the following commutative diagram
$$\xymatrix{
K(\varphi_\bullet)\ar[d]^{k_\bullet} & X_1 \ar[r]^{\left(
                                                                                        \begin{smallmatrix}
                                                                                          f_1  \\
                                                                                           \varphi_1\\
                                                                                        \end{smallmatrix}
                                                                                      \right)}\ar@{=}[d] & X_2\oplus Y_1 \ar[r]^{\left(
                                                                                        \begin{smallmatrix}
                                                                                          a_1 & -h_1 \\
                                                                                        \end{smallmatrix}
                                                                                      \right)}\ar[d]^{(1,0)
 } & Z \ar[r]^{f_3h_2}\ar[d]^{h_2} & \Sigma X_1\ar@{=}[d]\\
X_\bullet\ar[d]^{\pi_\bullet} & X_1 \ar[r]^{f_1}\ar[d]^{\varphi_1} & X_2 \ar[r]^{f_2}\ar@{-->}[d]^{a_1} & X_3 \ar[r]^{f_3}\ar@{=}[d] & \Sigma X_1 \ar[d]^{\Sigma\varphi_1}\\
I(\varphi_\bullet)\ar[d]^{i_\bullet} &  Y_1 \ar[r]^{h_1}\ar@{=}[d] & Z \ar[r]^{h_2}\ar@{-->}[d]^{a_2} & X_3 \ar[r]^{h_3}\ar[d]^{\varphi_3} & \Sigma Y_1\ar@{=}[d]\\
Y_\bullet\ar[d]^{c_\bullet} &  Y_1 \ar[r]^{g_1}\ar[d]^{h_1} & Y_2 \ar[r]^{g_2}\ar[d]^{\left(
                                                                                                                            \begin{smallmatrix}
                                                                                                                              0 \\
                                                                                                                              1 \\
                                                                                                                            \end{smallmatrix}
                                                                                                                          \right)} & Y_3 \ar[r]^{g_3}\ar@{=}[d] & \Sigma Y_1\ar[d]^{\Sigma h_1}\\
C(\varphi_\bullet) &   Z \ar[r]^{\left(
                                                                                        \begin{smallmatrix}
                                                                                          h_2  \\
                                                                                          a_2\\
                                                                                        \end{smallmatrix}
                                                                                      \right)} & X_3\oplus Y_2 \ar[r]^{(-\varphi_3, g_2)} & Y_3 \ar[r]^{(\Sigma h_1)g_3} & \Sigma Z\\
}$$
such that each row is a triangle and $\underline{\varphi_\bullet}=\underline{i_\bullet\pi_\bullet}$ in $\Delta(\mathcal{C})/\mathcal{R}_2$.
\end{lem}

\begin{proof}
We extend the morphism $h_3=\Sigma\varphi_1\cdot f_3:X_3\rightarrow\Sigma Y_1$ to a triangle $I(\varphi_\bullet)$.
By Lemma \ref{lem5.1} and its dual, we choose two morphisms $a_1:X_2\rightarrow Z$ and $a_2: Z\rightarrow Y_2$ such that the associated squares are commutative, and  $K(\varphi_\bullet)$ and $C(\varphi_\bullet)$ are triangles. We have $\underline{\varphi_\bullet}=\underline{i_\bullet\pi_\bullet}$ by  Proposition \ref{prop5.1}(a).
\end{proof}

\begin{thm}\label{thm5.2}
Let $\mathcal{C}$ be a triangulated category. Then the category $\Delta(\mathcal{C})/\mathcal{R}_2$  is an abelian category where the kernels and cokernels are given by homotopy cartesian diagrams.
\end{thm}

\begin{proof}
Given a morphism $\varphi_\bullet: X_\bullet\rightarrow Y_\bullet$ in $\Delta(\mathcal{C})$, as notations in Lemma \ref{lem5.2}, we can show that $K(\varphi_\bullet)$ is a kernel of $\underline{\varphi_\bullet}$, $C(\varphi_\bullet)$ is a cokernel of $\underline{\varphi_\bullet}$ and $I(\varphi_\bullet)$ is the image of $\underline{\varphi_\bullet}$.
\end{proof}

The following result is a triangulated analogue of Theorem \ref{thm4.4.1}.

\begin{thm}\label{thm5.3}
Let $\mathcal{C}$ be a Hom-finite Krull-Smidt $k$-linear triangulated category.

\textup{(a)} Assume that $X_\bullet: X_1\xrightarrow{f_1} X_2\xrightarrow{f_2} X_3\xrightarrow{f_3}\Sigma X_1$ is a triangle such that $f_3\neq 0$. Then $X_\bullet$ is a simple object in $\Delta(\mathcal{C})/\mathcal{R}_2$ if and only if $X_\bullet$ is an Auslander-Reiten triangle in $\mathcal{C}$.

\textup{(b)} There is a bijection between the set of isoclasses of simple objects in $\Delta(\mathcal{C})/\mathcal{R}_2$ and the set of isoclasses of Auslander-Reiten triangles in $\mathcal{C}$.
\end{thm}

%\begin{defn}
Let $\delta: X_1\xrightarrow{f_1} X_2\xrightarrow{f_2} X_3\xrightarrow{f_3}\Sigma X_1$ be a triangle. The {\em contravariant defect} $\delta^*$ and the {\em covariant defect} $\delta_*$ are defined by the following exact sequence of functors
$$\mathcal{C}(-,X_{1})\xrightarrow{\mathcal{C}(-,f_{1})}\mathcal{C}(-,X_{2})\xrightarrow{\mathcal{C}(-,f_{2})}\mathcal{C}(-,X_3)\rightarrow \delta^\ast\rightarrow 0,$$
$$\mathcal{C}(X_{3},-)\xrightarrow{\mathcal{C}(f_{2},-)}\mathcal{C}(X_{2},-)\xrightarrow{\mathcal{C}(f_{1},-)}\mathcal{C}(X_1,-)\rightarrow \delta_\ast\rightarrow 0.$$
%\end{defn}

\begin{example}
Let $\delta=P_X:\Sigma^{-1}X\rightarrow 0\rightarrow X\xrightarrow{1} X$, then $\delta^*=\mathcal{C}(-,X)$ and $\delta_*=\mathcal{C}(\Sigma^{-1}X,-)$.
\end{example}

\begin{rem}
Let $\mathcal{C}$ be a triangulated category. Then the equivalence $\Delta(\mathcal{C})/\mathcal{R}_2\cong \textup{mod-}\mathcal{C}$ is given by $\delta\mapsto\delta^*$ and the equivalence
$\Delta(\mathcal{C})/\mathcal{R}'_1\cong (\textup{mod-}\mathcal{C}^{\textup{op}})^{\textup{op}}$ is given by $\delta\mapsto\delta_*$. Since $\Delta(\mathcal{C})/\mathcal{R}_2=\Delta(\mathcal{C})/\mathcal{R}'_1$ by Proposition \ref{prop5.1}(a), we have a duality $$\phi: \textup{mod-}\mathcal{C}\rightarrow\textup{mod-}\mathcal{C}^{\textup{op}}, \delta^*\mapsto\delta_*.$$ By restriction, we have two dualities
$$\phi: \textup{proj-}\mathcal{C}\rightarrow\textup{inj-}\mathcal{C}^{\textup{op}},\ \ \mathcal{C}(-,X)\mapsto\mathcal{C}(\Sigma^{-1}X,-).$$
$$\phi: \textup{inj-}\mathcal{C}\rightarrow \textup{proj-}\mathcal{C}^{\textup{op}}, \ \ \mathcal{C}(-,\Sigma X)\mapsto\mathcal{C}(X,-).$$
Therefore, $\textup{mod-}\mathcal{C}$ is a Frobenius abelian category. So is $\Delta(\mathcal{C})/\mathcal{R}_2$. Moreover, each projective-injective object in $\Delta(\mathcal{C})/\mathcal{R}_2$ is of the form $X\rightarrow 0\rightarrow \Sigma X\xrightarrow{1} \Sigma X$.
\end{rem}

\begin{thm}\label{thm5.4}
Let $\mathcal{C}$ be a Hom-finite Krull-Smidt $k$-linear triangulated category. Assume that $\mathcal{C}$ is a dualizing $k$-variety. Then there is an equivalence $\tau:\mathcal{C}\cong \mathcal{C}$ such that
$D\mathcal{C}(\Sigma^{-1}X,-)\cong \mathcal{C}(-,\tau X)$ for each $X\in\mathcal{C}$ and $D\delta^*=\delta_*\tau$ for each triangle $\delta$.
\end{thm}

\begin{proof}
The composition of $\phi: \textup{mod-}\mathcal{C}\rightarrow\textup{mod-}\mathcal{C}^{\textup{op}}$ and $D:\textup{mod-}\mathcal{C}^{\textup{op}}\rightarrow\textup{mod-}\mathcal{C}$ is an equivalence $\theta=D\phi:\textup{mod-}\mathcal{C}\cong\textup{mod-}\mathcal{C}$. By restriction, we have an equivalence $\theta:\textup{proj-}\mathcal{C}\cong\textup{proj-}\mathcal{C}$. Since $\theta(\mathcal{C}(-,X))=D\mathcal{C}(\Sigma^{-1}X,-)\cong \mathcal{C}(-,Y)$ for some $Y\in\mathcal{C}$, there is an equivalence $\tau:\mathcal{C}\cong \mathcal{C}$ mapping $X$ to $Y$. In this case, $D\mathcal{C}(\Sigma^{-1}X,-)\cong \mathcal{C}(-,\tau X)$. Since $\tau$ induces an equivalence $\tau_\ast^{-1}:\textup{mod-}\mathcal{C}\cong\textup{mod-}\mathcal{C}$ such that $\theta=\tau_*^{-1}$, we have $D\delta_*=D\phi(\delta^*)=\tau^{-1}_*(\delta^*)=\delta^*\tau^{-1}$ for each triangle $\delta$. Thus $D\delta^*=\delta_*\tau$.
%Set $\delta=P_X$, then $D\mathcal{C}(-,X)\cong D\delta^*\cong\delta_*\tau\cong\mathcal{C}(\Sigma^{-1}X, \tau-)\cong\mathcal{C}(\Sigma^{-1}\tau^{-1}X,-)$.
\end{proof}

\begin{rem}
In Theorem \ref{thm5.4}, if we set $F=\tau\Sigma:\mathcal{C}\cong\mathcal{C}$, then $D\mathcal{C}(X,-)\cong\mathcal{C}(-,FX)$. Thus the functor $F$ is known as a Serre functor.
\end{rem}

\begin{rem}
%(a) Throughout this section, if we replace $\Delta(\mathcal{C})$ by a full subcategory $\mathcal{E}(\mathcal{C})$ given by a class of proper class of triangles in the sense of \cite{[Be]}, then we can obtain some similar results.

 Given an $n$-angulated category $\mathcal{C}$ in the sense of Geiss-Keller-Oppermann (see \cite{[GKO]}), one can consider the quotient of the category of $n$-angles and obtain some similar results. We leave them to the readers.
\end{rem}

\vspace{2mm}\noindent {\bf Acknowledgements} Most of the work was done when the author visited University of Stuttgart in 2017. He wants to thank Steffen Koenig for warm hospitality and for helpful discussions and remarks.

\end{document}